\documentclass[12pt]{amsart}
\usepackage{amsfonts}
\usepackage{amssymb}
\usepackage{amsmath}
\usepackage{amsthm}
\usepackage{mathrsfs}
\usepackage{tikz-cd}
\usepackage[capitalise, noabbrev]{cleveref}
\usepackage{latexsym, caption, subcaption, comment}
\usepackage{mdwlist,dirtytalk,times}
\usepackage{url, ragged2e}

\numberwithin{equation}{section}
\DeclareMathOperator{\supp}{supp}

\DeclareMathOperator{\Ind}{Ind}
\DeclareMathOperator{\del}{del}
\DeclareMathOperator{\lk}{link}
\DeclareMathOperator{\sta}{star}
\DeclareMathOperator{\depth}{depth}
\DeclareMathOperator{\dep}{depth}
\DeclareMathOperator{\pd}{pd}
\DeclareMathOperator{\reg}{reg}
\DeclareMathOperator{\sign}{sign}

\newcommand{\m}{\mathfrak{m}}

\newcommand{\tuple}[1]{\langle #1\rangle}

\newcommand{\Z}{\mathbb{Z}}

\newcommand{\kk}{\mathbb{K}}

\newcommand{\F}{\mathcal{F}}
\newcommand{\G}{\mathcal{G}}

\newcommand{\lr}[2]{(#1 \mid #2)}

\newtheorem{theorem}{Theorem}[section]
\newtheorem{lemma}[theorem]{Lemma}
\newtheorem{proposition}[theorem]{Proposition}
\newtheorem{corollary}[theorem]{Corollary}

\newtheorem{question}[theorem]{Question}

\theoremstyle{definition}
\newtheorem{definition}[theorem]{Definition} 
\newtheorem{remark}[theorem]{Remark}
\newtheorem{example}[theorem]{Example}

\newcommand{\qand}{\quad \mbox{and} \quad}
\newcommand{\qif}{\quad \mbox{if} \quad}

\newcommand{\qforsome}{\quad \mbox{for some} \quad}
\newcommand{\qotherwise}{\quad \mbox{otherwise} \quad}

\newcommand{\st}{\colon}
\newcommand{\qwhere}{\quad \mbox{where} \quad}

\numberwithin{equation}{section}
\begin{document}

\title[]{Spherical Complexes} 

\author[S. Faridi]{Sara Faridi}
\address[S. Faridi]
{Department of Mathematics \& Statistics,
Dalhousie University,
6297 Castine Way,
PO BOX 15000,
Halifax, NS,
Canada B3H 4R2
}
% \thanks{Faridi's research is supported by
% NSERC Discovery Grant 2023-05929}
\email{faridi@dal.ca}

\author[T. Holleben]{Thiago Holleben}
\address[T. Holleben]
{Department of Mathematics \& Statistics,
Dalhousie University,
6297 Castine Way,
PO BOX 15000,
Halifax, NS,
Canada B3H 4R2}
\email{hollebenthiago@dal.ca}


\begin{abstract} In this paper we define spherical complexes as simplicial complexes with the property that every subcomplex obtained by a sequence of links and deletions either has trivial homology, or has the homology of a sphere. Examples of such complexes are independence complexes of ternary graphs and independence complexes of simplicial forests.

   We give criteria for when a spherical complex is acyclic, and describe the dimension of the
    sphere when it is not. We then apply our results to compute the Leray number of these  complexes, and define combinatorial invariants for them which are counterparts to algebraic invariants of their Stanley-Reisner rings.
\end{abstract} 

\maketitle

%%%%%%%%%%%%%%%%%%%%%%%%%%%%%%%%%%%%%%%%%%%%%%%%%%%%%%%%%%%%   
\section{Introduction}
%%%%%%%%%%%%%%%%%%%%%%%%%%%%%%%%%%%%%%%%%%%%%%%%%%%%%%%%%%%%   

One of the most important classes of simplicial complexes in Combinatorics is the class of simplicial spheres: complexes that are homeomorphic to spheres. It is well known that for such complexes, there are strong restrictions imposed on the homology groups of special subcomplexes. More specifically, the links of every face of a simplicial sphere must have the homology groups of a sphere of the correct dimension.

Recently, Kim~\cite{kim} proved the following theorem, which was originally conjectured by Engstr\"om~\cite{engstrom}.

\begin{theorem}[Theorem 1.1, \cite{kim}]\label{ternaryintro}
    A graph $G$ has no induced cycles of length divisible by $3$ if and only if the independence complex of every induced subgraph of $G$ is either contractible or homotopy equivalent to a sphere.
\end{theorem}

By replacing the homotopy statements above to the weaker homology version, the topological condition on the independence complexes in \cref{ternaryintro} become similar to the one imposed on simplicial spheres. On one hand, the independence complex of \textit{every induced subcomplex} -- not only links, as is the case for simplicial spheres -- must have either the homology groups of a sphere, or have trivial homology groups. On the other hand, the only possible nonzero homology groups of links of faces of simplicial spheres are the top ones, while for the complexes in \cref{ternaryintro}, there is no such restriction.

More precisely, the complexes in~\cref{ternaryintro} have the property that every subcomplex that can be obtained by links and deletions either has trivial homology, or has the homology of a sphere. In this paper, we call such complexes \textit{spherical complexes}. 

Besides the independence complexes of ternary graphs, another large and well-known class of spherical complexes
is  Stanley-Reisner complexes of  simplicial forests (\cref{ex:ex-non-ex}).

The goal of this paper is to study properties that spherical complexes must satisfy, as a consequence of the homological conditions imposed by the definition.
We provide combinatorial methods to check whether a spherical complex is acyclic or not, and to check the dimension of the sphere when it is not. Our main result (\cref{t:topbetti})  gives us an upper and lower bound for the possible nonzero homologies of subcomplexes of $\Delta$ that can be obtained by a sequence of links and deletions. We state the result for induced subcomplexes and links of a spherical complex $\Delta$.

\begin{theorem}[See \cref{t:topbetti}]\label{t:topbettiintro}
    Let $\Delta$ be a spherical complex on $n$ vertices
    with $\tilde H_{d}(\Delta; \kk) \cong \kk$. 
   \begin{enumerate}
   \item If $\Sigma$ is an induced subcomplex of $\Delta$ on $m$ vertices such that  
     $\tilde H_{d'}(\Sigma, \kk) \cong \kk$, then
    $d - (n-m) \leq d' \leq d$.
 
    \item If  $\sigma \in \Delta$ is a face with $k$   vertices such that  
     $\tilde H_{d"}(\lk_{\Delta}(\sigma), \kk) \cong \kk$, then
    $d - k \leq d" \leq d$.
   \end{enumerate}
    In particular, the Leray number $L_\kk(\Delta) = d + 1$.
\end{theorem}

The \textit{Leray number} mentioned in \cref{t:topbettiintro} 
 is a topological invariant that appears naturally in the context of Helly type theorems, and can be seen as a measure of the complexity of $\Delta$. It is defined as 
$$
    L_\kk(\Delta) = \min \{d \st \tilde H_d(\Gamma, \kk) = 0  \text{ for all induced subcomplexes }\Gamma \subseteq \Delta\}.
$$
The Leray number of $\Delta$ has connections to the Castelnuovo-Mumford regularity of the Stanley-Reisner ideal of $\Delta$ in Commutative Algebra. 

In \cref{s:homology}
we introduce three combinatorial invariants -- sign, depth and projective dimension -- for a  spherical complex. The sign is used  to give a criterion for when  a spherical complex is acyclic in \cref{ternarytype}. 
 We show in~\cref{s:commalg} that the
names of the other two invariants are justified, as they are equal to their algebraic counterparts. 

We then apply our results to Commutative Algebra, and show that the homological restrictions imposed by spherical complexes imply the following relation between two a priori unrelated algebraic invariants.

\begin{theorem}[See \cref{depthreg}]\label{t:depthregintro}
    Let $\Delta$ be a non-acyclic spherical complex and $I_\Delta \subset R$ the Stanley-Reisner ideal of $\Delta$. Then 
    $$
        \reg(R/I_\Delta) = \depth(R/I_\Delta).
    $$
\end{theorem}

 \cref{t:depthregintro} not only gives an equality between two a priori unrelated invariants (i.e, there are examples where regularity can be less, equal or more than the depth), but also sheds some light into how different previous approaches on the study of different homological invariants of Stanley-Reisner ideals are related. 
    
 %%%%%%%%%%%%%%%%%%%%%%%%%%%%%%%%%%%%%%%%%%%%%%%%%%%%%%%%%%%%   
   \section{Preliminaries: spherical complexes}
%%%%%%%%%%%%%%%%%%%%%%%%%%%%%%%%%%%%%%%%%%%%%%%%%%%%%%%%%%%%   

    Let $V = \{v_1, \dots, v_n\}$ be a finite set. A
    \textbf{simplicial complex} over $V$ is a family of subsets
    $\Delta$ of $V$ such that $\sigma \in \Delta$ for every $\sigma
    \subseteq \tau \in \Delta$. Elements of $\Delta$ are called
    \textbf{faces} of $\Delta$ and maximal faces are called
    \textbf{facets}.  Given a set $A\subseteq V$, the \textbf{induced
      subcomplex} of $\Delta$ on $A$ is the simplicial complex
    $\Delta_A$ where $\tau \subseteq V$ is a face of $\Delta_A$ if and
    only if $\tau \subseteq A$ and $\tau \in \Delta$.

    A simplicial complex with facets $F_1,\ldots,F_q$ can be denoted by
    $\tuple{F_1,\ldots,F_q}$.
   
    We denote the {\bf void complex} with no faces by where $\{\}$. 
    
    We denote the simplicial complex that has the empty set as its only face as $\emptyset$.
    
    If $\Delta$ is a simplicial complex and $v$ is a vertex of $\Delta$, then \begin{itemize}
    \item $\Delta$ is a \textbf{cone} over $v$ if $\tau \cup v$ is a face of $\Delta$ for every face
    $\tau$ of $\Delta$;
  \item the {\bf deletion} of $v$ is the simplicial complex 
 $\del_\Delta(v)=\{\sigma \in \Delta \st v \notin \sigma\}$;
  
  if $w$ is a vertex not in $\Delta$, then we set $\del_\Delta(w)=\Delta$;
   \item the {\bf link} of $v$ is the simplicial complex $\lk_\Delta(v) = \{\sigma \in \del_\Delta(v) \st   \sigma \cup \{v \} \in \Delta \}$;
    
    if $w$ is a vertex not in $\Delta$, then we set $\lk_\Delta(w)=\{\}$;
   \item the {\bf star} of $v$ is the simplicial complex~$\sta_\Delta(v) = \{\sigma \in \Delta \st \sigma \cup \{v\} \in \Delta\}$;
   \item the \textbf{Alexander dual} of $\Delta$ is the simplicial complex
    $\Delta^\vee = \{V \setminus \tau \st \tau \not \in \Delta\}$, and has the following property for a field $\kk$ and  all $i$
    $$\dim_\kk \tilde H_{n - i - 2}(\Delta;\kk) = \dim_\kk \tilde H_{i - 1}(\Delta^\vee;\kk).$$
   \item  The  \textbf{independence complex} $\Ind(\Delta)$ of $\Delta$ is the  simplicial complex whose faces are the independent sets of $\Delta$, where an {\bf independent set} of $\Delta$ is a subset $A$ of  the vertex set of $\Delta$ such that if $\sigma \subseteq A$, $\sigma$ does not contain a facet of $\Delta$ (note that this generalizes the independence complex of a graph when the graph is considered as a one-dimensional simplicial complex.

\end{itemize}

\begin{example}\label{example0}
   A graph $H$ and its (contractible) independence complex $\Ind(H)$ appear below.

\begin{center}
\begin{tabular}{cc}
        \begin{tikzpicture}[x=0.75pt,y=0.75pt,yscale=-1,xscale=1]
       %Straight Lines [id:da7366247837916307] 
        \draw    (208.29,141.02) -- (320.29,141.02) ;
        \draw [shift={(320.29,141.02)}, rotate = 0] [color={rgb, 255:red, 0; green, 0; blue, 0 }  ][fill={rgb, 255:red, 0; green, 0; blue, 0 }  ][line width=0.75]      (0, 0) circle [x radius= 3.35, y radius= 3.35]   ;
        \draw [shift={(208.29,141.02)}, rotate = 0] [color={rgb, 255:red, 0; green, 0; blue, 0 }  ][fill={rgb, 255:red, 0; green, 0; blue, 0 }  ][line width=0.75]      (0, 0) circle [x radius= 3.35, y radius= 3.35]   ;
        %Straight Lines [id:da8601820719119107] 
        \draw    (208.29,141.02) -- (190.29,91.02) ;
        \draw [shift={(190.29,91.02)}, rotate = 250.2] [color={rgb, 255:red, 0; green, 0; blue, 0 }  ][fill={rgb, 255:red, 0; green, 0; blue, 0 }  ][line width=0.75]      (0, 0) circle [x radius= 3.35, y radius= 3.35]   ;
        \draw [shift={(208.29,141.02)}, rotate = 250.2] [color={rgb, 255:red, 0; green, 0; blue, 0 }  ][fill={rgb, 255:red, 0; green, 0; blue, 0 }  ][line width=0.75]      (0, 0) circle [x radius= 3.35, y radius= 3.35]   ;
        %Straight Lines [id:da9095929835839343] 
        \draw    (302.29,92.02) -- (320.29,141.02) ;
        \draw [shift={(320.29,141.02)}, rotate = 69.83] [color={rgb, 255:red, 0; green, 0; blue, 0 }  ][fill={rgb, 255:red, 0; green, 0; blue, 0 }  ][line width=0.75]      (0, 0) circle [x radius= 3.35, y radius= 3.35]   ;
        \draw [shift={(302.29,92.02)}, rotate = 69.83] [color={rgb, 255:red, 0; green, 0; blue, 0 }  ][fill={rgb, 255:red, 0; green, 0; blue, 0 }  ][line width=0.75]      (0, 0) circle [x radius= 3.35, y radius= 3.35]   ;
        %Straight Lines [id:da7124282060888232] 
        \draw    (350.29,95.02) -- (320.29,141.02) ;
        \draw [shift={(320.29,141.02)}, rotate = 123.11] [color={rgb, 255:red, 0; green, 0; blue, 0 }  ][fill={rgb, 255:red, 0; green, 0; blue, 0 }  ][line width=0.75]      (0, 0) circle [x radius= 3.35, y radius= 3.35]   ;
        \draw [shift={(350.29,95.02)}, rotate = 123.11] [color={rgb, 255:red, 0; green, 0; blue, 0 }  ][fill={rgb, 255:red, 0; green, 0; blue, 0 }  ][line width=0.75]      (0, 0) circle [x radius= 3.35, y radius= 3.35]   ;

        % Text Node
        \draw (178,75) node [anchor=north west][inner sep=0.75pt]   [align=left] {$\displaystyle a$};
        % Text Node
        \draw (196,145) node [anchor=north west][inner sep=0.75pt]   [align=left] {$\displaystyle b$};
        % Text Node
        \draw (322.29,144.02) node [anchor=north west][inner sep=0.75pt]   [align=left] {$\displaystyle c$};
        % Text Node
        \draw (289,73) node [anchor=north west][inner sep=0.75pt]   [align=left] {$\displaystyle d$};
        % Text Node
        \draw (352,77) node [anchor=north west][inner sep=0.75pt]   [align=left] {$\displaystyle e$};
        \end{tikzpicture}
$\quad$ & $\quad$ 
        \begin{tikzpicture}[x=0.75pt,y=0.75pt,yscale=-1,xscale=1]
        
        %Shape: Triangle [id:dp9750611615572666] 
        \draw  [fill={rgb, 255:red, 155; green, 155; blue, 155 }  ,fill opacity=1 ][line width=0.75]  (286,115.91) -- (321,156) -- (251,156) -- cycle ;
        %Shape: Triangle [id:dp7576908594086484] 
        \draw  [fill={rgb, 255:red, 155; green, 155; blue, 155 }  ,fill opacity=1 ] (285.95,195.96) -- (251,156) -- (321,156) -- cycle ;
        %Straight Lines [id:da9832961767394821] 
        \draw    (286,116) -- (345,156) ;
        %Shape: Triangle [id:dp9352073369567431] 

          \draw [shift={(345,156)}, rotate = 250.2] [color={rgb, 255:red, 0; green, 0; blue, 0 }  ][fill={rgb, 255:red, 0; green, 0; blue, 0 }  ][line width=0.75]      (0, 0) circle [x radius= 3.35, y radius= 3.35]   ;

  \draw [shift={(286,115.91)}, rotate = 250.2] [color={rgb, 255:red, 0; green, 0; blue, 0 }  ][fill={rgb, 255:red, 0; green, 0; blue, 0 }  ][line width=0.75]      (0, 0) circle [x radius= 3.35, y radius= 3.35]   ;

    \draw [shift={(321,156)}, rotate = 250.2] [color={rgb, 255:red, 0; green, 0; blue, 0 }  ][fill={rgb, 255:red, 0; green, 0; blue, 0 }  ][line width=0.75]      (0, 0) circle [x radius= 3.35, y radius= 3.35]   ;

      \draw [shift={(251,156)}, rotate = 250.2] [color={rgb, 255:red, 0; green, 0; blue, 0 }  ][fill={rgb, 255:red, 0; green, 0; blue, 0 }  ][line width=0.75]      (0, 0) circle [x radius= 3.35, y radius= 3.35]   ;

        \draw [shift={(285.95,195.96)}, rotate = 250.2] [color={rgb, 255:red, 0; green, 0; blue, 0 }  ][fill={rgb, 255:red, 0; green, 0; blue, 0 }  ][line width=0.75]      (0, 0) circle [x radius= 3.35, y radius= 3.35]   ;

        % Text Node
        \draw (270,110) node [anchor=north west][inner sep=0.75pt]    [align=left] {$\displaystyle a$};
        % Text Node
        \draw (270,190) node [anchor=north west][inner sep=0.75pt]   [align=left] {$\displaystyle b$};
        % Text Node
        \draw (348, 160) node [anchor=north west][inner sep=0.75pt]   [align=left] {$\displaystyle c$};
        % Text Node
        \draw (233,160) node [anchor=north west][inner sep=0.75pt]   [align=left] {$\displaystyle d$};
        % Text Node
        \draw (328,160) node [anchor=north west][inner sep=0.75pt]   [align=left] {$\displaystyle e$};
        % Text Node

        \end{tikzpicture}
\end{tabular}
\end{center}
\end{example}

\begin{example}\label{example1}
    A complex $\Delta$ and its (non-acyclic) independence complex $\Ind(\Delta)$ appear below.

    \tikzset{every picture/.style={line width=0.75pt}} %set default line width to 0.75pt        
\begin{center}
\begin{tikzpicture}[x=0.75pt,y=0.75pt,yscale=-1,xscale=1]

%Shape: Triangle [id:dp04074510796151132] 
\draw  [fill={rgb, 255:red, 155; green, 155; blue, 155 }  ,fill opacity=1 ] (136,111) -- (171,151) -- (101,151) -- cycle ;
%Straight Lines [id:da6798884838401136] 
\draw    (136,111) -- (136.82,77.82) ;
\draw [shift={(136.82,77.82)}, rotate = 271.41] [color={rgb, 255:red, 0; green, 0; blue, 0 }  ][fill={rgb, 255:red, 0; green, 0; blue, 0 }  ][line width=0.75]      (0, 0) circle [x radius= 3.35, y radius= 3.35]   ;
\draw [shift={(136,111)}, rotate = 271.41] [color={rgb, 255:red, 0; green, 0; blue, 0 }  ][fill={rgb, 255:red, 0; green, 0; blue, 0 }  ][line width=0.75]      (0, 0) circle [x radius= 3.35, y radius= 3.35]   ;
%Straight Lines [id:da21566468855115506] 
\draw    (171,151) -- (101,151) ;
\draw [shift={(101,151)}, rotate = 180] [color={rgb, 255:red, 0; green, 0; blue, 0 }  ][fill={rgb, 255:red, 0; green, 0; blue, 0 }  ][line width=0.75]      (0, 0) circle [x radius= 3.35, y radius= 3.35]   ;
\draw [shift={(171,151)}, rotate = 180] [color={rgb, 255:red, 0; green, 0; blue, 0 }  ][fill={rgb, 255:red, 0; green, 0; blue, 0 }  ][line width=0.75]      (0, 0) circle [x radius= 3.35, y radius= 3.35]   ;
%Shape: Triangle [id:dp688343839328863] 
\draw  [fill={rgb, 255:red, 155; green, 155; blue, 155 }  ,fill opacity=1 ] (296,109) -- (331,149) -- (261,149) -- cycle ;
%Straight Lines [id:da060467379849228786] 
\draw    (261,149) -- (296.82,75.82) ;
\draw [shift={(296.82,75.82)}, rotate = 296.08] [color={rgb, 255:red, 0; green, 0; blue, 0 }  ][fill={rgb, 255:red, 0; green, 0; blue, 0 }  ][line width=0.75]      (0, 0) circle [x radius= 3.35, y radius= 3.35]   ;
\draw [shift={(261,149)}, rotate = 296.08] [color={rgb, 255:red, 0; green, 0; blue, 0 }  ][fill={rgb, 255:red, 0; green, 0; blue, 0 }  ][line width=0.75]      (0, 0) circle [x radius= 3.35, y radius= 3.35]   ;
%Straight Lines [id:da668660888623922] 
\draw    (331,149) -- (261,149) ;
\draw [shift={(261,149)}, rotate = 180] [color={rgb, 255:red, 0; green, 0; blue, 0 }  ][fill={rgb, 255:red, 0; green, 0; blue, 0 }  ][line width=0.75]      (0, 0) circle [x radius= 3.35, y radius= 3.35]   ;
\draw [shift={(331,149)}, rotate = 180] [color={rgb, 255:red, 0; green, 0; blue, 0 }  ][fill={rgb, 255:red, 0; green, 0; blue, 0 }  ][line width=0.75]      (0, 0) circle [x radius= 3.35, y radius= 3.35]   ;
%Straight Lines [id:da44588302130618973] 
\draw    (331,149) -- (296.82,75.82) ;
\draw [shift={(296.82,75.82)}, rotate = 244.96] [color={rgb, 255:red, 0; green, 0; blue, 0 }  ][fill={rgb, 255:red, 0; green, 0; blue, 0 }  ][line width=0.75]      (0, 0) circle [x radius= 3.35, y radius= 3.35]   ;
\draw [shift={(331,149)}, rotate = 244.96] [color={rgb, 255:red, 0; green, 0; blue, 0 }  ][fill={rgb, 255:red, 0; green, 0; blue, 0 }  ][line width=0.75]      (0, 0) circle [x radius= 3.35, y radius= 3.35]   ;
%Straight Lines [id:da21442294545872986] 
\draw    (331,149) -- (296,109) ;
\draw [shift={(296,109)}, rotate = 228.81] [color={rgb, 255:red, 0; green, 0; blue, 0 }  ][fill={rgb, 255:red, 0; green, 0; blue, 0 }  ][line width=0.75]      (0, 0) circle [x radius= 3.35, y radius= 3.35]   ;
\draw [shift={(331,149)}, rotate = 228.81] [color={rgb, 255:red, 0; green, 0; blue, 0 }  ][fill={rgb, 255:red, 0; green, 0; blue, 0 }  ][line width=0.75]      (0, 0) circle [x radius= 3.35, y radius= 3.35]   ;

% Text Node
\draw (86,152) node [anchor=north west][inner sep=0.75pt]   [align=left] {$\displaystyle x$};
% Text Node
\draw (173,154) node [anchor=north west][inner sep=0.75pt]   [align=left] {$\displaystyle y$};
% Text Node
\draw (145,100) node [anchor=north west][inner sep=0.75pt]   [align=left] {$\displaystyle z$};
% Text Node
\draw (132,56) node [anchor=north west][inner sep=0.75pt]   [align=left] {$\displaystyle w$};
% Text Node
\draw (246,150) node [anchor=north west][inner sep=0.75pt]   [align=left] {$\displaystyle x$};
% Text Node
\draw (333,152) node [anchor=north west][inner sep=0.75pt]   [align=left] {$\displaystyle y$};
% Text Node
\draw (290,90) node [anchor=north west][inner sep=0.75pt]   [align=left] {$\displaystyle w$};
% Text Node
\draw (292,54) node [anchor=north west][inner sep=0.75pt]   [align=left] {$\displaystyle z$};

\end{tikzpicture}
\end{center}
\end{example}

 \begin{definition}[{\bf Spherical complex}]
    For a field $\kk$ a simplicial complex $\Delta$ on vertex set $V$ is called {\bf $\kk$-spherical} if 
    \begin{enumerate}
        \item $\tilde{H}_i(\Delta;\kk)=
        \begin{cases} 1 & \mbox{ for at most one value of } i\\
        0 & \mbox{otherwise};
        \end{cases}$
        \medskip
        \item for every $v \in V$ the two complexes $\lk_\Delta(v)$ and $\del_\Delta(v)$ are $\kk$-spherical.
    \end{enumerate}
\end{definition}

\begin{example}[{\bf Examples and non-examples of spherical complexes}]\label{ex:ex-non-ex}

\medskip

\begin{enumerate}
\item A large class of spherical complexes is the family of independence complexes of ternary graphs~\cite{kim}. A graph is called a {\bf ternary graph} if it does not have induced cycles of length divisible by $3$. Ternary graphs are spherical since the ternary property is preserved by taking induced subgraphs, and for independence complexes of graphs, taking links corresponds to taking special induced subgraphs. The graph $H$ in \cref{example0} is an example of a ternary graph, it is in fact a tree.

\item Another large class is independence complexes of {\em simplicial forests}. In~\cite{NF2014}, Erey and Faridi showed that such complexes  either have trivial homology groups, or the homology groups of  spheres. Moreover, the link and deletion of simplicial forests are again simplicial forests~(\cite[Lemma~7]{simplicialforests}, \cite[Lemma~1]{facet-ideal}). The complex $\Delta$ in~\cref{example1} is an example of a simplicial forest.

\item In order to find non-examples of spherical complexes, we can take the independence complex of any graph that contains a cycle of length divisible by $3$. Take for example the graph $R_3$ in \cref{f:sns}, which is the complement of a hexagon. 
To see that $\Ind(R_3)$ is not spherical, take the subcomplex $\Sigma$ obtained by deleting the vertices $2,4,6$. Since $\Sigma$ consists of $3$ points, it has 
    $$
        \tilde H_0(\Sigma) \cong \Z \oplus \Z.
    $$
Note that $\Ind(R_3)$ is a simplicial sphere that is not a spherical complex.

\begin{figure}
\tikzset{every picture/.style={line width=0.75pt}} %set default line width to 0.75pt        

\begin{tikzpicture}[x=0.75pt,y=0.75pt,yscale=-1,xscale=1]

%Shape: Rectangle [id:dp5275946027574203] 
\draw   (79,54) -- (171,54) -- (171,94) -- (79,94) -- cycle ;
%Straight Lines [id:da3223820013963352] 
\draw    (171,54) -- (137,75.4) ;
\draw [shift={(137,75.4)}, rotate = 147.81] [color={rgb, 255:red, 0; green, 0; blue, 0 }  ][fill={rgb, 255:red, 0; green, 0; blue, 0 }  ][line width=0.75]      (0, 0) circle [x radius= 3.35, y radius= 3.35]   ;
\draw [shift={(171,54)}, rotate = 147.81] [color={rgb, 255:red, 0; green, 0; blue, 0 }  ][fill={rgb, 255:red, 0; green, 0; blue, 0 }  ][line width=0.75]      (0, 0) circle [x radius= 3.35, y radius= 3.35]   ;
%Straight Lines [id:da4718491290857565] 
\draw    (171,94) -- (137,75.4) ;
\draw [shift={(137,75.4)}, rotate = 208.68] [color={rgb, 255:red, 0; green, 0; blue, 0 }  ][fill={rgb, 255:red, 0; green, 0; blue, 0 }  ][line width=0.75]      (0, 0) circle [x radius= 3.35, y radius= 3.35]   ;
\draw [shift={(171,94)}, rotate = 208.68] [color={rgb, 255:red, 0; green, 0; blue, 0 }  ][fill={rgb, 255:red, 0; green, 0; blue, 0 }  ][line width=0.75]      (0, 0) circle [x radius= 3.35, y radius= 3.35]   ;
%Straight Lines [id:da20727857583268205] 
\draw    (108,75.4) -- (79,54) ;
\draw [shift={(79,54)}, rotate = 216.42] [color={rgb, 255:red, 0; green, 0; blue, 0 }  ][fill={rgb, 255:red, 0; green, 0; blue, 0 }  ][line width=0.75]      (0, 0) circle [x radius= 3.35, y radius= 3.35]   ;
\draw [shift={(108,75.4)}, rotate = 216.42] [color={rgb, 255:red, 0; green, 0; blue, 0 }  ][fill={rgb, 255:red, 0; green, 0; blue, 0 }  ][line width=0.75]      (0, 0) circle [x radius= 3.35, y radius= 3.35]   ;
%Straight Lines [id:da6071563859247813] 
\draw    (108,75.4) -- (79,94) ;
\draw [shift={(79,94)}, rotate = 147.32] [color={rgb, 255:red, 0; green, 0; blue, 0 }  ][fill={rgb, 255:red, 0; green, 0; blue, 0 }  ][line width=0.75]      (0, 0) circle [x radius= 3.35, y radius= 3.35]   ;
\draw [shift={(108,75.4)}, rotate = 147.32] [color={rgb, 255:red, 0; green, 0; blue, 0 }  ][fill={rgb, 255:red, 0; green, 0; blue, 0 }  ][line width=0.75]      (0, 0) circle [x radius= 3.35, y radius= 3.35]   ;
%Straight Lines [id:da19998324996922556] 
\draw    (137,75.4) -- (108,75.4) ;
\draw [shift={(108,75.4)}, rotate = 180] [color={rgb, 255:red, 0; green, 0; blue, 0 }  ][fill={rgb, 255:red, 0; green, 0; blue, 0 }  ][line width=0.75]      (0, 0) circle [x radius= 3.35, y radius= 3.35]   ;
\draw [shift={(137,75.4)}, rotate = 180] [color={rgb, 255:red, 0; green, 0; blue, 0 }  ][fill={rgb, 255:red, 0; green, 0; blue, 0 }  ][line width=0.75]      (0, 0) circle [x radius= 3.35, y radius= 3.35]   ;
%Shape: Regular Polygon [id:dp05948089495624598] 
\draw   (340,74) -- (327.5,95.65) -- (302.5,95.65) -- (290,74) -- (302.5,52.35) -- (327.5,52.35) -- cycle ;

% Text Node
\draw (110,123) node [anchor=north west][inner sep=0.75pt]   [align=left] {$\displaystyle R_{3}$};
% Text Node
\draw (67,35) node [anchor=north west][inner sep=0.75pt]   [align=left] {$\displaystyle 1$};
% Text Node
\draw (130,55) node [anchor=north west][inner sep=0.75pt]   [align=left] {$\displaystyle 2$};
% Text Node
\draw (67,97) node [anchor=north west][inner sep=0.75pt]   [align=left] {$\displaystyle 3$};
% Text Node
\draw (173,36) node [anchor=north west][inner sep=0.75pt]   [align=left] {$\displaystyle 4$};
% Text Node
\draw (103,55) node [anchor=north west][inner sep=0.75pt]   [align=left] {$\displaystyle 5$};
% Text Node
\draw (173,97) node [anchor=north west][inner sep=0.75pt]   [align=left] {$\displaystyle 6$};
% Text Node
\draw (292,36) node [anchor=north west][inner sep=0.75pt]   [align=left] {$\displaystyle 1$};
% Text Node
\draw (328,36) node [anchor=north west][inner sep=0.75pt]   [align=left] {$\displaystyle 2$};
% Text Node
\draw (342,66) node [anchor=north west][inner sep=0.75pt]   [align=left] {$\displaystyle 3$};
% Text Node
\draw (328.5,98.65) node [anchor=north west][inner sep=0.75pt]   [align=left] {$\displaystyle 4$};
% Text Node
\draw (292,98.65) node [anchor=north west][inner sep=0.75pt]   [align=left] {$\displaystyle 5$};
% Text Node
\draw (275,66) node [anchor=north west][inner sep=0.75pt]   [align=left] {$\displaystyle 6$};
% Text Node
\draw (299,123) node [anchor=north west][inner sep=0.75pt]   [align=left] {$\displaystyle \Ind(R_3)$};
\end{tikzpicture}
\caption{A simplicial sphere that is not spherical (\cref{ex:ex-non-ex})}\label{f:sns}
\end{figure}
\end{enumerate}

\end{example}

Our main topological tool is the following well-known link-deletion Mayer-Vietoris sequence. For a proof on how to get the long exact sequence in \cref{t:mvseq} see~\cite[Theorem 2.2]{linkdel}.

\begin{theorem}[\textbf{Link-deletion Mayer-Vietoris sequence}]\label{t:mvseq}
    Let $\Delta$ be a simplicial complex and $v$ a vertex of $\Delta$. Then we have the following long exact sequence:
    $$
        \cdots \rightarrow \tilde H_{i}(\lk_\Delta( v)) \rightarrow
            \tilde H_{i}(\del_\Delta(v)) \rightarrow
        \tilde H_{i}(\Delta) \rightarrow \tilde H_{i-1}(\lk_\Delta( v)) \rightarrow \cdots
    $$
    Moreover, the same long exact sequence holds if the coefficients of the homology groups are taken to be a field $\kk$.
\end{theorem}

\begin{lemma}\label{l:diagramchasing}
    Let $\Delta$ be a $\kk$-spherical complex and $v$ a vertex. Then at least one of the following holds 
    \begin{enumerate}
        \item $\tilde H_j(\Delta; \kk) \cong \tilde H_j(\del_\Delta(v); \kk)$ and $\tilde H_j(\lk_\Delta(v); \kk) = 0$ for every $j$,
        \item $\tilde H_j(\Delta; \kk) \cong \tilde H_{j - 1}(\lk_\Delta(v); \kk)$ and $\tilde H_j(\del_\Delta(v); \kk) = 0$ for every $j$ or
        \item $\tilde H_j(\lk_\Delta(v); \kk) \cong \tilde H_j(\del_\Delta(v); \kk)$ and $\tilde H_j(\Delta; \kk) = 0$ for every $j$
    \end{enumerate}
\end{lemma}

\begin{proof}
    By the definition of spherical complex, we have two possible cases. If $\Delta$ is acyclic, then~\cref{t:mvseq} implies 
        \begin{equation}\label{e:mvj}
         \tilde H_j(\lk_\Delta(v); \kk) \cong \tilde H_j(\del_\Delta(v); \kk)    
        \end{equation}
        for every $j$, which means that~(3) holds.
        
       Suppose  $\Delta$ is not acyclic, and assume without loss of generality that $\Delta$ has the homology of a $k$-dimensional sphere. In other words, 
       $$\tilde H_j(\Delta; \kk) \neq 0 \iff j = k, 
       \qand 
       \tilde H_k(\Delta; \kk) = \kk.$$
             \cref{t:mvseq} gives us the following exact sequence 
   \begin{equation}\label{e:small-mv}
   {\small
     \begin{array}{rl}
        0
        \rightarrow
        \tilde H_{k}(\lk_{\Delta}(v); \kk) 
        \rightarrow
       \tilde H_{k}(\del_{\Delta}(v); \kk) 
 \rightarrow
        &\kk\\
        \rightarrow&
        \tilde H_{k-1}(\lk_{\Delta}(v); \kk)
         
         \rightarrow 
        \tilde H_{k-1}(\del_{\Delta}(v); \kk) 
        \rightarrow 
         0 
         \end{array}}
     \end{equation}    
     and also \eqref{e:mvj} for all $j \notin \{k,k-1\}$. Since both $\lk_\Delta(v)$ and $\del_\Delta(v)$ are either acyclic, or have the homology of a sphere,     
    and  the alternating sum of dimensions must add up to zero, we may simplify the sequence above to
    $$ \cdots \rightarrow 
        0 
        \rightarrow
        \tilde H_{k}(\del_{\Delta}(v); \kk) 
        \rightarrow
        \kk 
        \rightarrow 
        \tilde H_{k-1}(\lk_{\Delta}(v); \kk) 
        \rightarrow
        0 
        \rightarrow
        \cdots
       $$
    which implies exactly one of (1) or (2) holds.

\end{proof}

%%%%%%%%%%%%%%%%%%%%%%%%%%%%%%%%%%%%%%%%%%%%%%%%%%%%%%%%%%%%%%%%%
\section{Filtrations of spherical complexes}\label{s:filtrations}\label{sec-filtration}
%%%%%%%%%%%%%%%%%%%%%%%%%%%%%%%%%%%%%%%%%%%%%%%%%%%%%%%%%%%%%%%%%
We are now ready to set up the required tools to find the  homology of induced subcomplexes and links of spherical complexes. Some of the notation we use  previously appeared  in the literature for studying independence complexes of ternary graphs~\cite{kim,wuzhang}.

\begin{definition}[$\Delta\lr{X}{Y}$]
\label{d:xy} Given a sequence of vertices $v_1, \dots, v_n$ of a simplicial complex $\Gamma$ partitioned into disjoint sets $X$ and $Y$, and a subcomplex $\Delta$ of $\Gamma$, we define the following sequence of subcomplexes.
\begin{itemize}
\item $\Delta_1=\Delta$; 
\item if $i>1$, then  
$$ \Delta_{i+1}=
\begin{cases}
 \lk_{\Delta_i}(v_i) & \qif v_i \in X \\
 \del_{\Delta_i}(v_i) & \qif v_i \in Y. 
\end{cases}
$$  
\end{itemize}

    We then denote $\Delta_{n + 1}$ by $\Delta\lr{X}{Y}$ and similarly if 
    $$
        X_i = X \cap \{v_1, \dots, v_i\} \qand Y_i = Y \cap \{v_1, \dots, v_i\},
    $$
    then we denote $\Delta_{i + 1}$ by $\Delta\lr{X_i}{Y_i}$. In particular, 
    $\Delta\lr{\emptyset}{\emptyset}=\Delta$, and for a vertex $v$ of $\Gamma$ we have 
    $$
    \Delta \lr{\emptyset}{v} = \del_\Delta(v) 
    \qand 
    \Delta \lr{v}{\emptyset} =  \lk_\Delta(v).
    $$
\end{definition}

We may then define the following function analogous to the one in Kim's work~\cite{kim}.

\begin{definition}[$d\lr{X}{Y}$]\label{d:d}
    Let $\Delta$ be  a $\kk$-spherical complex, and $X,Y$ as in \cref{d:xy}. We set
        $$
        d\lr{X}{Y} = \begin{cases}
            k \qif \tilde H_k(\Delta\lr{X}{Y}; \kk) \cong \kk \qforsome k\\
            \ast \qotherwise
        \end{cases}
        $$
  where $\ast$ is just a symbol.      
 \end{definition}

We now define filtrations, these are the main objects we use in order to prove~\cref{t:topbettiintro}.

\begin{definition}[{\bf Filtrations}]\label{d:filtration}
    Let $\Delta$ be a $\kk$-spherical complex with the homology of a $k$-dimensional sphere, or equivalently, $d\lr{\emptyset}{\emptyset} = k \neq \ast$. A \textbf{filtration} $\F$ of $\Delta$ is a sequence of subcomplexes
    $$
        \F := \Delta_1 \supseteq  \Delta_2 \supseteq  \dots \supseteq  \Delta_q
    $$
    with a sequence of vertices $v_1 \in \Delta_1, \dots, v_{q - 1} \in \Delta_{q - 1}$ such that
    \begin{enumerate}
        \item $\Delta_1 = \Delta$;
        \item $\Delta_{i + 1} \in \{\del_{\Delta_i}(v_i), \lk_{\Delta_i}(v_i)\}$ for every $i$;
        \item $\Delta_{i + 1}$ has the homology of a sphere for every $i$ when the coefficients are taken to be $\kk$.
    \end{enumerate}
    We say a filtration is \textbf{maximal} if $\Delta_q = \emptyset$.

    Since each $\Delta_i$ is not acyclic, by \cref{l:diagramchasing} for every $i$, there is $t_i \in \{0, 1\}$ such that
    $$
        \tilde H_j(\Delta_i; \kk) \cong \tilde H_{j - t_i}(\Delta_{i + 1}; \kk) \quad \mbox{ for all }j,
\qwhere 
    \Delta_{i + 1}  = 
    \begin{cases}
        \del_{\Delta_i}(v_i) & \qif t_i=0\\
        \lk_{\Delta_i}(v_i) & \qif t_i=1.
    \end{cases}
    $$
\end{definition}

A consequence of \cref{t:mvseq,l:diagramchasing} is the following.

\begin{proposition}[\textbf{Existence of filtrations}]\label{p:filtrationsexist}
    Let $\Delta$ be a $\kk$-spherical complex. If $\Delta$ is not acyclic, then it has a maximal filtration.
\end{proposition}

\begin{proof}
    Let $v_1 \in \Delta = \Delta_1$, and assume $\Delta$ has the homology of a $k$-dimensional sphere. By \cref{t:mvseq} we have the  exact sequence in \eqref{e:small-mv}.
    % $${\small
    % 0 \to \tilde H_{k}(\lk_\Delta( v_1); \kk) \to \tilde H_k(\del_\Delta(v_1); \kk) \to \kk \to \tilde H_{k - 1}(\lk_\Delta( v_1); \kk) \to H_{k - 1}(\del_\Delta(v_1); \kk) \to 0.}
    % $$
    Moreover, by the definition of spherical complex, both $\lk_\Delta( v_1)$ and $\del_\Delta(v_1)$ either have trivial homology, or have the homology of a sphere. In particular, by~\cref{l:diagramchasing} exactly one of the following holds:
    \begin{enumerate}
        \item $\lk_\Delta( v)$ has the homology of $S^{k - 1}$ and $\del_\Delta(v)$ has trivial homology.
        \item $\del_\Delta(v)$ has the homology of $S^k$ and $\lk_\Delta( v)$ has trivial homology.
    \end{enumerate}
    Note that taking $\Delta_2$ to be the complex above without trivial homology gives us the first step in a filtration. Inductively, since every simplicial complex obtained by taking links and deletions of $\Delta$ is also spherical by definition, we can repeat the process by choosing a new vertex $v_2 \in \Delta_2$. After a finite number of steps, the complex $\Delta_q$ will have only $\emptyset$ as a face, and in this case $\Delta_q = S^{-1}$, so the result holds.
\end{proof}

\cref{p:filtrationsexist} also implies any sequence of vertices $v_1, \dots, v_s$ such that $\{v_1, \dots, v_s\} \in \Delta$ can be used to build a filtration.

\begin{lemma}[{\bf Analogue of Lemma 3.2~ \cite{kim}}]\label{triangles}
    Let $\Delta$ be a spherical complex and $X, Y \subseteq V(\Delta)$ such that $X \cap Y = \emptyset$. For every vertex $v \not \in X \cup Y$, the triple 
    $$
    \big ( d\lr{X}{Y}, d\lr{X \cup \{v\}}{Y}, d\lr{X}{Y \cup \{v\}} \big )
    \in \big \{    (\ast, \ast, \ast), (k, \ast, k), (k + 1, k, \ast), (\ast, k, k) \big \}
    $$
    for some integer $k \geq -1$.
\end{lemma}

\begin{proof}
  We split the proof in two cases.

  \begin{enumerate}
      \item If $v \not \in \Delta \lr{X}{Y}$, then $\Delta \lr{X \cup \{v\}}{Y} = \{\}$ and $\Delta \lr{X}{Y \cup \{v\}} = \Delta \lr{X}{Y}$, hence 
      $$
      \big ( d\lr{X}{Y}, d\lr{X \cup \{v\}}{Y}, d\lr{X}{Y \cup \{v\}} \big ) \in \{(\ast, \ast, \ast), (k, \ast, k)\}
      $$
      \item If $v \in \Delta \lr{X}{Y}$, the result follows by~\cref{l:diagramchasing}
  \end{enumerate}
  
\end{proof}

We are now ready to prove our main result.

\begin{theorem}[{\bf Main Theorem}]\label{t:topbetti}
    Let $\Delta$ be a spherical complex with the homology of a sphere of dimension $d_\Delta$. Let $\Sigma$ be a subcomplex of $\Delta$ such that $\Sigma = \Delta\lr{X}{Y}$, where $X, Y$ are disjoint subsets of the vertex set of $\Delta$. Then $\Sigma$ either has no homology or has the homology of a sphere of dimension $d_\Sigma$, where
    $$
        d_\Delta - |X| - |Y| \leq d_\Sigma \leq d_\Delta.
    $$
\end{theorem}

\begin{proof}
    Fix $\Sigma = \Delta\lr{A}{B}$ a subcomplex of $\Delta$. Since $\Delta$ is a spherical complex, we know $\Sigma$ either has no homology, or has the homology of a sphere of dimension $d_\Sigma$. If $\Sigma$ has no homology, we have nothing to show.  In the latter case, we proceed by induction on $n = |A| + |B|$ to show 
    $$
        d_\Delta - |A| - |B| \leq d_\Sigma \leq d_\Delta.
    $$
    If $n = 1$, then either $A = \{v\}$ and $B = \emptyset$ or $A = \emptyset$ and $B = \{v\}$ for some $v \in \Delta$. The result follows directly by \cref{triangles} since
    $$
        \big{(}d\lr{X}{Y}, d\lr{X \cup \{v\}}{Y}, d\lr{X}{Y \cup \{v\}}\big{)} \in \big{\{}(d_\Delta,\ast,d_\Delta), (d_\Delta, d_\Delta - 1, \ast)\big{\}}. 
    $$
    Now suppose $n > 1$ and the statement holds for every subcomplex $\Delta\lr{A'}{B'}$ with $|A'| + |B'| \leq n$. Fix two subsets $A = \{v_1, \dots, v_q\}$, $B = \{v_{q + 1}, \dots, v_{n + 1}\}$ such that $A \cap B = \emptyset$, and let $d_\Sigma = d\lr{A}{B} \neq \ast$. By \cref{triangles}, for every $v_i \in A$ we have
    $$
        (d\lr{A \setminus \{v_i\}}{B}, d\lr{A}{B}, d\lr{A \setminus \{v_i\}}{B \cup \{v_i\}}) \in \{(d_\Sigma + 1, d_\Sigma, \ast), (\ast, d_\Sigma, d_\Sigma)\}.
    $$
    If the triple is $(d_\Sigma + 1, d_\Sigma, \ast$), then by the induction hypothesis applied to the complex $\Delta\lr{A \setminus \{v_i\}}{B}$ we have 
    $$
        d_\Delta - n \leq d_\Sigma + 1 \leq d_\Delta \implies d_\Delta - (n + 1) \leq d_\Sigma \leq d_\Delta
    $$
    and we are done.

    Similarly, for every $v_i \in B$, by \cref{triangles} we have
    $$
        (d\lr{A}{B \setminus \{v_i\}}, d\lr{A \cup \{v_i\}}{B \setminus \{v_i\}}, d\lr{A}{B}) \in \{(d_\Sigma, \ast, d_\Sigma), (\ast, d_\Sigma, d_\Sigma)\}
    $$
    If the triple is $(d_\Sigma, \ast, d_\Sigma)$, by the induction hypothesis applied to the complex $\Delta\lr{A}{B \setminus \{v_i\}}$ we have 
    $$
        d_\Delta - n \leq d_\Sigma \leq d_\Delta \implies d_\Delta - (n + 1) \leq d_\Sigma \leq d_\Delta
    $$
    and we are done.
    
    In summary, the only way for the induction step to fail, is if for every $v_i \in A$ and $v_j \in B$
    \begin{align*}\label{e:ab}
        (d\lr{A \setminus \{v_i\}}{B}, d\lr{A}{B}, d\lr{A \setminus \{v_i\}}{B \cup \{v_j\}}) & = (\ast, d_\Sigma, d_\Sigma) \\
        (d\lr{A}{B \setminus \{v_j\}}, d\lr{A \cup \{v_j\}}{B \setminus \{v_j\}}, d\lr{A}{B}) & = (\ast, d_\Sigma, d_\Sigma)
    \end{align*}    
    but if that is the case, then for every $C, D \subseteq A \cup B$ with $C \cup D = A \cup B$ and $C \cap D = \emptyset$ we have 
    $$
        d\lr{C}{D} = d\lr{A}{B} = d_\Sigma \neq \ast.
    $$
    Moreover, by picking $u \in C$ (or $v \in D$), by setting $X = C \setminus \{u\}$ and $Y = D$ (or $X = C$ and $Y = D \setminus \{v\})$ in \cref{triangles} we can conclude that
    $$
        d\lr{C \setminus \{u\}}{D} = d\lr{C}{D \setminus \{v\}} = \ast.
    $$
    Now since $\Delta$ is not acyclic, we can apply the same argument from~\cref{p:filtrationsexist} to show there is a filtration 
    $$
        \F: \Delta = \Delta_1 \supseteq \Delta_2 \supseteq \dots \supseteq \Delta_n \supseteq \Delta_{n + 1}
    $$
    of $\Delta$ using the order of the vertices $v_1, \dots, v_{n + 1}$ of $A \cup B$. Note that this is possible since $d\lr{\emptyset}{\emptyset} \neq \ast$.

    At the last step of the filtration we have a contradiction because for some $C', D'$ with $C' \cup D' = (A \cup B) \setminus \{v_{n + 1}\}$ and $C' \cap D' = \emptyset$, we have $d\lr{C'}{D'} \neq \ast$ and $\Delta_n = \Delta\lr{C'}{D'}$. 
\end{proof}

\begin{definition}
    The {\bf $\kk$-Leray number} of a simplicial complex $\Delta$ is defined as 
    $$
    L_\kk(\Delta) = \min \{d \st \tilde H_d(\Sigma; \kk) = 0  \quad \text{for all induced subcomplexes } \Sigma \subseteq \Delta\}.
    $$
    In other words, $L_\kk(\Delta)$ is the minimum $d$ such that every induced subcomplex $\Sigma$ of $\Delta$ (including $\Delta$ itself) has trivial $d$-homology.
\end{definition}

A direct consequence of~\cref{t:topbetti} is the following.

\begin{corollary}\label{c:lerayspherical}
    Let $\Delta$ be a $\kk$-spherical complex with $\tilde H_d(\Delta; \kk) \cong \kk$. Then
    $$
        L_\kk(\Delta) = d + 1.
    $$
\end{corollary}

\begin{proof}
    Let $\Sigma$ be an induced subcomplex of $\Delta$. Then 
    $$
        \Sigma = \Delta \lr{\emptyset}{V(\Delta) \setminus V(\Sigma)}.
    $$ Since $\Delta$ is spherical, we know $\Sigma$ is either acyclic, or has the homology of a $d'$-dimensional sphere. \cref{t:topbetti} then implies 
    $$
        d' \leq d,
    $$
    so the Leray number of $\Delta$ is $d + 1$.
\end{proof}

\begin{example}
    Let $\Delta$ be the following complex.
    \begin{center}

\tikzset{every picture/.style={line width=0.75pt}} %set default line width to 0.75pt        

\begin{tikzpicture}[x=0.75pt,y=0.75pt,yscale=-1,xscale=1]

%Shape: Square [id:dp7662059588312344] 
\draw  [fill={rgb, 255:red, 155; green, 155; blue, 155 }  ,fill opacity=1 ] (205,103) -- (255,103) -- (255,153) -- (205,153) -- cycle ;
%Straight Lines [id:da3649660130526797] 
\draw    (205,103) -- (255,153) ;
%Straight Lines [id:da7883060353458078] 
\draw    (255,103) -- (205,153) ;
%Straight Lines [id:da4552191347637704] 
\draw    (292,103) -- (291.82,154.82) ;
\draw [shift={(291.82,154.82)}, rotate = 90.2] [color={rgb, 255:red, 0; green, 0; blue, 0 }  ][fill={rgb, 255:red, 0; green, 0; blue, 0 }  ][line width=0.75]      (0, 0) circle [x radius= 3.35, y radius= 3.35]   ;
\draw [shift={(292,103)}, rotate = 90.2] [color={rgb, 255:red, 0; green, 0; blue, 0 }  ][fill={rgb, 255:red, 0; green, 0; blue, 0 }  ][line width=0.75]      (0, 0) circle [x radius= 3.35, y radius= 3.35]   ;
%Straight Lines [id:da5896905864119157] 
\draw    (205,103) -- (205,153) ;
\draw [shift={(205,153)}, rotate = 90] [color={rgb, 255:red, 0; green, 0; blue, 0 }  ][fill={rgb, 255:red, 0; green, 0; blue, 0 }  ][line width=0.75]      (0, 0) circle [x radius= 3.35, y radius= 3.35]   ;
\draw [shift={(205,103)}, rotate = 90] [color={rgb, 255:red, 0; green, 0; blue, 0 }  ][fill={rgb, 255:red, 0; green, 0; blue, 0 }  ][line width=0.75]      (0, 0) circle [x radius= 3.35, y radius= 3.35]   ;
%Straight Lines [id:da8099048347964328] 
\draw    (255,103) -- (255,153) ;
\draw [shift={(255,153)}, rotate = 90] [color={rgb, 255:red, 0; green, 0; blue, 0 }  ][fill={rgb, 255:red, 0; green, 0; blue, 0 }  ][line width=0.75]      (0, 0) circle [x radius= 3.35, y radius= 3.35]   ;
\draw [shift={(255,103)}, rotate = 90] [color={rgb, 255:red, 0; green, 0; blue, 0 }  ][fill={rgb, 255:red, 0; green, 0; blue, 0 }  ][line width=0.75]      (0, 0) circle [x radius= 3.35, y radius= 3.35]   ;
%Straight Lines [id:da6069049551118255] 
\draw    (230,128) -- (205,153) ;
\draw [shift={(205,153)}, rotate = 135] [color={rgb, 255:red, 0; green, 0; blue, 0 }  ][fill={rgb, 255:red, 0; green, 0; blue, 0 }  ][line width=0.75]      (0, 0) circle [x radius= 3.35, y radius= 3.35]   ;
\draw [shift={(230,128)}, rotate = 135] [color={rgb, 255:red, 0; green, 0; blue, 0 }  ][fill={rgb, 255:red, 0; green, 0; blue, 0 }  ][line width=0.75]      (0, 0) circle [x radius= 3.35, y radius= 3.35]   ;

% Text Node
\draw (226,131) node [anchor=north west][inner sep=0.75pt]   [align=left] {$\displaystyle 1$};
% Text Node
\draw (257,83) node [anchor=north west][inner sep=0.75pt]   [align=left] {$\displaystyle 2$};
% Text Node
\draw (257,156) node [anchor=north west][inner sep=0.75pt]   [align=left] {$\displaystyle 3$};
% Text Node
\draw (192,156) node [anchor=north west][inner sep=0.75pt]   [align=left] {$\displaystyle 4$};
% Text Node
\draw (194,83) node [anchor=north west][inner sep=0.75pt]   [align=left] {$\displaystyle 5$};
% Text Node
\draw (294,80) node [anchor=north west][inner sep=0.75pt]   [align=left] {$\displaystyle 6$};
% Text Node
\draw (293.82,157.82) node [anchor=north west][inner sep=0.75pt]   [align=left] {$\displaystyle 7$};

\end{tikzpicture}
    \end{center}

    Note that the simplicial complex $\Sigma \subseteq \Delta$ induced on the set of vertices $\{2,3,4,5\}$ has $\tilde H_1(\Sigma; \kk) \cong \kk$ for any field $\kk$, while $\tilde H_0(\Delta; \kk) \cong \kk$ and $\tilde H_j(\Delta; \kk) = 0$ for every $j > 0$. It can be shown that $L_\kk(\Delta) = 2$, and in particular, by~\cref{c:lerayspherical},  $\Delta$ is not spherical. Indeed, the complex $\Gamma \subseteq \Delta$ induced on $\{3,5,6\}$ has $\tilde H_0(\Gamma; \kk) \cong \kk \oplus \kk$.
\end{example}
%%%%%%%%%%%%%%%%%%%%%%%%%%%%%%%%%%%%%%%%%%%%%%%%%%%%%%%
\section{The homology, depth and projective dimension  of spherical complexes}\label{s:homology} 
%%%%%%%%%%%%%%%%%%%%%%%%%%%%%%%%%%%%%%%%%%%%%%%%%%%%%%%

We define three invariants of spherical complexes -- sign, depth and projective dimension -- in this section. The latter two, as we will see in~\cref{s:commalg}, are equal to their algebraic counterparts, so that the names will be justified. The sign is used  to give a criterion for when  a spherical complex is acyclic in \cref{ternarytype}. 

\begin{definition}\label{d:GB}
    Let $\Delta$ be a $\kk$-spherical complex and $v_{i_1}, \dots, v_{i_s}$ an ordered subset $B$ of the vertices of $\Delta$. We define the poset $\G_B$ to be the set 
    $$
        \{\lr{C}{D} \st C \cup D = \{v_{i_1}, \dots, v_{i_j}\} \subseteq  B \qand C \cap D = \emptyset \}
    $$
    with a partial order $\leq$ given by $\lr{C'}{D'} \leq \lr{C}{D}$ if and only if $C \subseteq  C'$ and $D \subseteq D'$. We view the poset $\G_B$ by its Hasse diagram, which is a binary tree.

   Let $j_{\G_B}$ be the number of pairs $\lr{C}{D}$ in $\G_B$ such that $d\lr{C}{D} \neq \ast$ and $C\cup D=B$,
  and define  the \textbf{sign} of $\G_B$ to be $$\sign_{\G_B}= (-1)^{j_{\G_B}}.$$   

\end{definition}

\begin{lemma}\label{paritystep}
    Let $\Delta$ be a spherical complex, $B = \{v_{i_1}, \dots, v_{i_s}\}$ an ordered set of vertices of $\Delta$ and $\G_B$ be a poset as in~\cref{d:GB}. Let $f(t)$ denote the number of pairs $\lr{C}{D}$ in $\G_B$ such that $d\lr{C}{D} \neq \ast$ and $C \cup D = \{v_{i_1}, \dots, v_{i_t}\}$. Then for every $0 \leq t, t' \leq |B|$ we have $f(t)$ is even if and only if $f(t')$ is even.
\end{lemma} 

\begin{proof}
    Assume  $\lr{C}{D} \in \G_B$ with $C \cup D =\{v_1,\ldots,v_t\}$. By \cref{triangles}, the possible values of the triple 
    \begin{equation}\label{e:cd}
    \big ( d\lr{C}{D}, \ d\lr{C \cup \{v_{t + 1}\}}{D}, \  d\lr{C}{D \cup \{v_{t + 1}\}} \big )
    \end{equation}
    are
    $$
        (\ast, \ast, \ast), \quad (k + 1, k, \ast), \quad (k, \ast, k), \quad (\ast, k, k)
    $$
    for some integer $d \geq -1$. 

     If $d\lr{C}{D} \neq \ast$, the pair $\lr{C}{D}$ contributes exactly $1$  pair $\lr{C'}{D'}$ with $d\lr{C'}{D'} \neq \ast$ and $C'\cup D'= \{v_{i_1}, \dots, v_{i_{t+1}}\}$. 
    If $d\lr{C}{D} = \ast$, the pair $\lr{C}{D}$ contributes $0$ or $2$ pairs $\lr{C'}{D'}$ with $d\lr{C'}{D'} \neq \ast$ and $C'\cup D'= \{v_{i_1}, \dots, v_{i_{t+1}}\}$.
    The conclusion now follows.
\end{proof}

 \begin{theorem}[{\bf The sign of $\Delta$}]\label{parity}
        Let $\Delta$ be a spherical complex, $B$ a set of vertices and let $\G_B$ be a poset as in \cref{d:GB}. Then 
        \begin{enumerate}
             \item $\sign_{\G_B}$ is independent of the ordering of the vertices of $B$.
            \item $\sign_{\G_B} = \sign_{\G_{B'}}$ for any other set of vertices $B'$.
            \item  $\sign_{\G_B}$ is $1$ if and only if $\Delta$ is acyclic.
            \end{enumerate}
\end{theorem}

\begin{proof}
    Using (the notation in) \cref{paritystep}, if $B'$ is  any  set of vertices, with any ordering,  
       we have
        $$
        \sign_{\G_B}= (-1)^{f(|B|)} =(-1)^{f(0)} = (-1)^{f(|B'|)} = \sign_{\G_{B'}}
       $$
       which settles~(1) and~(2). Item~(3) also follows immediately, after observing that 
       $$f(0)=\begin{cases} 1 &  \qif d\lr{\emptyset}{\emptyset} \neq \ast \ \mbox{ (i.e. $\Delta$ is not acyclic)} \\
                            0 &  \qif d\lr{\emptyset}{\emptyset} = \ast \ \mbox{ (i.e. $\Delta$ is acyclic)}.
        \end{cases}
       $$ 
\end{proof}

\begin{definition}
    Let $\Delta$ be a non-acyclic spherical complex and a maximal filtration $\F$
    $$
    \Delta = \Delta_1 \supseteq \dots \supseteq \Delta_q = \emptyset
    $$
     of $\Delta$. Recalling that each step in the filtration $\F$ corresponds to a vertex deletion or taking a
    link, we will use the following notation.
     \begin{enumerate}
        \item  $\del(\F) = \{v_i \st \Delta_{i + 1} = \del_{\Delta_i}(v_i)\}$
        \item  $\lk(\F) = \{v_i \st \Delta_{i + 1} = \lk_{\Delta_i} (v_i)\}$
        \item The \textbf{depth} of $\F$ is $\dep(\F) = |\lk(\F)|$
        \item $N(\F) = {\displaystyle V(\Delta) \setminus (\bigcup_{v \in \lk(\F)}V(\sta_\Delta(v)))}$.
        % \item $N(\F) = \cup_{v \in \lk(\F)} N(v)$. {\color{red} need to fix this one}
    \end{enumerate}
\end{definition}

\begin{theorem}[{\bf Computing homology of a spherical complex}]\label{ternarytype}
    Let $\Delta$ be a spherical complex.
    \begin{enumerate}
        \item $\Delta$ is acyclic if and only if $\sign(\Delta) = 1$. 
        \item  If $\sign(\Delta) = -1$ then   $\Delta$ has the homology of a $(\dep(\F) - 1)$--dimensional sphere, where $\F$ is any maximal filtration of $\Delta$, and in particular 
        $\dep(\F)$ is an invariant of $\Delta$.
    \end{enumerate}
\end{theorem}

\begin{proof}
    In view of \cref{ternarytype}, we only need to prove that for a non-acyclic spherical complex $\Delta$ and a maximal filtration $\F$ of $\Delta$, $\Delta$ has the homology of a $(\dep(\F) - 1)$--dimensional sphere. Let $\F$ be a maximal filtration  of $\Delta$. By \cref{l:diagramchasing}, each time $v_i \in \del(\F)$, then
    $$
        \tilde H_j(\Delta_{i}; \kk) \cong \tilde H_j(\Delta_i; \kk).
    $$
    Again by \cref{l:diagramchasing}, when $v_i \in \lk(\F)$ we have the following isomorphisms for all $i$
    $$
        \tilde H_{j - 1}(\Delta_{i + 1}); \kk) \cong \tilde H_{j}(\Delta_{i}; \kk).
    $$
    Finally, the empty complex $\Delta_q=\Delta\lr{\lk(\F)}{\del(\F)}$ and so $\Delta_q$ has the homology of $S^{-1}$. We have
    $$
        \tilde H_{\dep(\F) - 1}(\Delta; \kk) \cong \tilde H_{-1}(\Delta\lr{\lk(\F)}{\del(\F)}; \kk) \cong \kk,
    $$
    and the result follows.
\end{proof}

%\section{Depth and projective dimension of spherical complexes}

Before stating the next lemma, we note that given a subcomplex $\Sigma = \Delta \lr{X}{Y} \neq \{\}$ of $\Delta$, the set $X$ must be a face of $\Delta$. Indeed, given any face $\sigma$ of $\Sigma$ (including the only $(-1)$--dimensional face of $\Sigma$), we know by the definition of link that $\sigma \cup X \in \Delta$.

\begin{lemma}\label{l:combinatorialauslander}
    Let $\Delta$ be a non-acyclic spherical complex and $\F$ a maximal filtration of $\Delta$. Then
    $$
        |V(\Delta)| = \dep(\F) + |\del(\F)| + |N(\F)| - |\del(\F) \cap N(\F)|.
    $$
\end{lemma}

\begin{proof}
    Let $(v_1, \dots, v_s)$ be the sequence of vertices that defines the maximal filtration $\F$. Then by definition we have $\Delta\lr{\lk(\F)}{\del(\F)} = \emptyset$. In particular
    $$
        V(\Delta) = \lk(\F) \cup N(\F) \cup \del(\F)
    $$
    and since $\lk(\F)$ is a face of $\Delta$, we know $\lk(\F) \cap N(\F) = \emptyset$. We also know that $\lk(\F)$ is disjoint 
from $\del(\F)$. So in the expression $|\lk(\F)| + |\del(\F)| + |N(\F)|$, the only vertices being counted more than once are the ones in $\del(\F) \cap N(\F)$.
\end{proof}

Let $\Delta$ be a non-acyclic spherical complex $\Delta$. In the spirit of the \textit{Auslander-Buchsbaum formula} from commutative algebra (see \cite[Chapter 1]{cmrings}), we define the \textbf{projective dimension} of a maximal filtration $\F$ of $\Delta$ as:
$$
    \pd(\F) = |\del(\F)| + |N(\F)| - |\del(\F) \cap N(\F)|.
$$

\begin{corollary}[{\bf Depth and projective dimension as combinatorial invariants}]
    Let $\Delta$ be a non-acyclic spherical complex, and  let $\F, \F'$ be two maximal filtrations of $\Delta$. Then 
    
    \begin{enumerate}
        \item $\dep(\F) = \dep(\F')$,
        \item $\pd(\F)=\pd(\F')$.
    \end{enumerate}
\end{corollary}
 
\begin{proof}
    This follows directly from  \cref{ternarytype} and  \cref{l:combinatorialauslander}. 
\end{proof}

In light of the last results we give the following definitions, that are the combinatorial counterparts of the algebraic invariants (see \cref{s:commalg}).

\begin{definition}[{\bf Depth and projective dimension of spherical complexes}]\label{d:homological}
    Let $\Delta$ be a spherical complex. If $\Delta$ is not acyclic, let $\F$ be any maximal filtration of $\Delta$. We call $$\pd(\Delta) = \pd(\F) \qand \dep(\Delta) = \dep(\F)$$ the  
    \textbf{projective dimension} and \textbf{depth} of $\Delta$, respectively.

    If $\Delta$ is acyclic, we set $$\pd(\Delta) = \max\{\pd(\Sigma) \st \Sigma \mbox{ induced subcomplex of $\Delta$ and $\Sigma$ not acyclic} \}$$ and $$\dep(\Delta) = |V(\Delta)| - \pd(\Delta).$$
\end{definition}

The following statement is a combinatorial version of the well-known Auslander-Buchsbaum formula from commutative algebra, which follows now directly from \cref{l:combinatorialauslander} and \cref{d:homological}.

\begin{theorem}[\textbf{Auslander-Buchsbaum for spherical complexes}]\label{combinatorialauslander}
    Let $\Delta$ be a spherical complex. Then
    $$
        |V(\Delta)| = \dep(\Delta) + \pd(\Delta).
    $$
\end{theorem}

%%%%%%%%%%%%%%%%%%%%%%%%%%%%%%%%%%%%%%%%%%%%%%%%%%%%%%%
\section{Applications to commutative algebra}\label{s:commalg}
%%%%%%%%%%%%%%%%%%%%%%%%%%%%%%%%%%%%%%%%%%%%%%%%%%%%%%%

Our original motivation for the topological investigations in this paper was computing algebraic invariants of squarefree monomial ideals. This section is devoted to translating the earlier results of the paper into the language of algebra. Throughout let $R = \kk[x_1, \dots, x_n]$ be the polynomial ring in $n$ variables. 

 Given a simplicial complex $\Delta = \tuple{F_1, \dots, F_s}$ with vertex set $\{1, \dots, n\}$, the ideal 
    $$
        I_\Delta = (x_{i_1}\dots x_{i_r} \st  \{i_1, \dots, i_r\} \not \in \Delta) \subseteq R
    $$
    is called the \textbf{Stanley-Reisner} ideal of $\Delta$.

    A sequence of polynomials $f_1, \dots, f_n \in (x_1, \dots, x_n)$ is a \textbf{regular sequence} for $R/I_\Delta$ if 
    \begin{itemize}
        \item[-] $f_i$ is a nonzero divisor of $S/\big ( (f_1, \dots, f_{i - 1})+ I_\Delta\big )$ for all $i$, and 
        \item[-] $R/\big ((f_1, \dots, f_n)+ I_\Delta\big ) \neq 0$.
    \end{itemize}
    The \textbf{depth} of $R/I_\Delta$ is the length of a maximal regular sequence for $R/I_\Delta$.

\begin{example}[\textbf{Filtrations and regular sequences}, Example 1.4 in \cite{MR4093701}]
    Let $R = \kk[x_1, \dots, x_8]$ and $I(G)$ the Stanley-Reisner ideal of the independence complex $\Delta$ of the following graph $G$ (in other words, $I(G)$ is the edge ideal of $G$).

    \begin{center}
    \begin{tikzpicture}

        \tikzstyle{point}=[circle,thick,draw=black,fill=black,inner sep=0pt,minimum width=4pt,minimum height=4pt]
    
     \node (a)[point,label={[xshift=-0.4cm, yshift=-.10 cm]: $1$}] at (5,2.5) {};
    
        \node (b)[point,,label={[xshift=-0.4cm, yshift=-.10 cm]:$2$}] at (3,2.5) {};
    
    \node (d)[point,,label={[xshift=0cm, yshift=0.1 cm]: $4$}] at (4,0.5) {};
      \node(c)[point,,label={[xshift=-0.4cm, yshift=-.20 cm]:$3$}] at (3,1.5) {};
    
       \node (e)[point,label={[xshift=-0.4cm, yshift=-.10 cm]:$5$}] at (5,1.5) {};
    
        \node (f)[point,label={[label distance=0cm]3:$6$}] at (6,2) {};
    
    \node (g)[point,,label={[xshift=0 cm, yshift=0 cm]: $7$}] at (6,1) {};
      \node(h)[point,,label={[xshift=0.4cm, yshift=-0.3 cm]:$8$}] at (7,.5) {};
    
    \draw (a.center) -- (b.center) -- (c.center) -- (d.center) -- (e.center) -- (a.center);
    \draw (a.center) -- (f.center);
    \draw (e.center) -- (g.center) -- (h.center);
    \end{tikzpicture}
\end{center}

In \cite[Example 1.4]{MR4093701}, the authors apply their results to show that 
$$
    x_6 + x_1, x_8 + x_7, x_4 + x_5 + x_3
$$
forms a $R/I(G)$-regular sequence.

 We now show that the regular sequence above can be understood from the point of view of filtrations. Observe that $\Delta\lr{\emptyset}{1}=\Delta -\{1\}$ is the independence complex of a forest with an isolated vertex, and hence is acyclic. On the other hand, $\Delta \lr{1}{\emptyset}$ is the independence complex of the graph below, and in particular $\Delta \lr{1}{\emptyset}$ is not acyclic.  
\begin{center}
    \begin{tikzpicture}

        \tikzstyle{point}=[circle,thick,draw=black,fill=black,inner sep=0pt,minimum width=4pt,minimum height=4pt]
    
    \node (d)[point,,label={[xshift=0cm, yshift=0.1 cm]: $4$}] at (4,0.5) {};
      \node(c)[point,,label={[xshift=-0.4cm, yshift=-.20 cm]:$3$}] at (3,1.5) {};
    
     \node (g)[point,,label={[xshift=0 cm, yshift=0 cm]: $7$}] at (6,1) {};
      \node(h)[point,,label={[xshift=0.4cm, yshift=-0.3 cm]:$8$}] at (7,.5) {};
    
    \draw (c.center) -- (d.center);
    \draw (g.center) -- (h.center);
    \end{tikzpicture}
\end{center}

The discussion above implies the following equality 
$$
    (\Delta\lr{\emptyset}{\emptyset}, \Delta\lr{1}{\emptyset}, \Delta\lr{\emptyset}{1}) = (2, 1, \ast).
$$
  In particular, $\Delta$ has a maximal filtration by~\cref{p:filtrationsexist}. The following is a maximal filtration of $\Delta$:
 $$
    \F: \ \ \Delta \supseteq  \Delta \lr{6}{\emptyset} \supseteq  \Delta\lr{6}{2} \supseteq  \Delta\lr{\{6, 8\}}{2} \supseteq  \Delta\lr{\{6, 8, 4\}}{2} = \emptyset.
$$
Note that $$x_6 + x_1 = \sum_{i \in N[6]} x_i, \quad x_8 + x_7 = \sum_{i \in N[8]} x_i, \quad  x_3 + x_4 + x_5 = \sum_{i \in N[4]} x_i.$$

In other words, the sum of neighbors (in $G$) of each element in $\lk(\F)$ gives us the maximal regular sequence from \cite{MR4093701} that realizes the depth of $R/I(G)$. This pattern seems to suggest that filtrations might be an effective tool to build regular sequences, please see \cref{r:filtrations-star-packing} and   \cref{q:star-packing}.
\end{example}

A {\bf free resolution} of $R/I_\Delta$ is an exact sequence of free $R$-modules encoding the  relations between the generators of $I_\Delta$,    
$$                                                              
0 \to R^{a_r} \stackrel{d_r}{\longrightarrow} \cdots \stackrel{d_2}{\longrightarrow} R^{a_1} \stackrel{d_1}{\longrightarrow} \
R^{a_0}          
  $$
where the image of $d_1$ is $I_\Delta$. Alternatively, a free resolution of $R/I_\Delta$ can be represented by an exact sequence of $\kk$-vector spaces of the same dimensions as the free modules above.
The smallest possible such sequence, or in other words the sequence
 with the smallest possible values for the ranks $a_i$, is
the (unique up to isomorphism of complexes) {\bf minimal free resolution } of $I_\Delta$
\begin{equation}\label{e:mfr}
  0 \to R^{\beta_p} \stackrel{d_p}{\longrightarrow} \cdots
  \stackrel{d_2}{\longrightarrow} R^{\beta_1}
  \stackrel{d_1}{\longrightarrow} R^{\beta_0}.
\end{equation}
The minimal free resolution \eqref{e:mfr} contains many of the algebraic invariants of $R/I_\Delta$, including the 
{\bf betti numbers} $\beta_i$ and the
{\bf projective dimension} $p$ of $R/I_\Delta$.

The betti numbers $\beta_i$ of $R/I_\Delta$ can be computed using a {\bf multigrading}
$$    \beta_i = \sum_\m { \beta_{i, \m}}$$
    where $\m$ runs over the monomials of $R$,  and the numbers $\beta_{i, \m}(R/I_\Delta)$, called the \textbf{multigraded betti numbers}, can be calculated as dimensions of homology spaces of induced subcomplexes of $\Delta$ (\cref{simplicialresolutions}). 

    Finally, the \textbf{Castelnuovo-Mumford regularity} of $R/I_\Delta$ is 
    $$
        \reg(R/I_\Delta) := \max (\deg \m - i \st   \beta_{i, \m} \neq 0).
    $$
 For more details on free resolutions we refer the reader to~\cite{cmrings}.

The main result that allows us to apply results from previous sections to the study of minimal free resolutions is Hochster's formula (\cref{simplicialresolutions}). To state this result, we first note that for a monomial $\m$, the \textbf{support} of $\m$ is the set $\supp(\m) = \{i \st x_i \mid \m\}$. 
If $\Delta$ is a simplicial complex with vertex set $[n]$, we use $\Delta_\m$ to denote the induced subcomplex $\Delta \lr{\emptyset}{[n] \setminus \supp(\m)}$. %Moreover, it is not difficult to see that $$\Ind(G_\m)=\Ind(G)_\m = \{\sigma \in \Ind(G) \st \sigma \subseteq \supp(\m)\},$$ where $\Ind(G)_\m$ is the induced subcomplex of $\Ind(G)$ on $\supp(\m)$.

\begin{theorem}[\textbf{Hochster's formula}, \cite{hochsterformula}]\label{simplicialresolutions}
    Let $R = \kk[x_1, \dots, x_n]$ and $\Delta$ a simplicial complex with vertex set $[n]$. Then for a squarefree monomial $\m \in R$
      $$
         \beta_{i, \m}(R/I_\Delta) = \dim_{\kk} \tilde H_{\deg \m - i - 1}(\Delta_{\m}; \kk).
      $$ Moreover, $\beta_{i, \m}(R/I_\Delta) = 0$ if $\m$ is not squarefree.
\end{theorem}

\begin{theorem}[\textbf{Betti numbers of spherical complexes}]\label{mainthm}
    Let $\Delta$ be a $\kk$-spherical complex and  $\m$ any squarefree monomial in $R = \kk[x_1, \dots, x_n]$.  Then
    the following are equivalent.
    \begin{enumerate}
    \item $\beta_{i, \m}(R/I_\Delta)\neq 0$ 
    \item $\beta_{i, \m}(R/I_\Delta)=1$. 
    \item $i=\pd(\Delta_\m)$ and $\Delta_\m$ is not acyclic 
    \item $i=\pd(\Delta_\m)$ and $\sign(\Delta_\m)=-1$.
    \end{enumerate}
   \end{theorem}

\begin{proof}
    By definition, for every induced subcomplex $\Delta_\m$ of $\Delta$, the
    homology groups of $\Delta_\m$ are either $0$ or one (and only one)
    homology group is $\kk$. By \cref{ternarytype}, $\Delta_\m$ has a
    nonzero homology group if and only if $\sign(\Delta_\m) = -1$, moreover,
    the nonzero homology group is $\tilde H_{\dep(\Delta_\m) - 1}(\Delta_\m; \kk)
    \cong \kk$. By \cref{combinatorialauslander} we have
    $$
        |V(\Delta_\m)| = \pd(\Delta_\m) + \dep(\Delta_\m)
    $$
    and so we can rewrite the isomorphism of the homology group as: 
    $$
        \tilde H_{|V(\Delta_\m)| - \pd(\Delta_\m) - 1}(\Delta_\m; \kk) \cong \kk.
    $$
    In particular we have:

    \begin{equation}\label{formula}
        \tilde H_{|V(\Delta_\m)| - i - 1}(\Delta_\m; \kk) \cong \begin{cases}
            \kk \ \ \ \text{if $\sign(\Delta_\m) = -1$ and $\pd(\Delta_\m) = i$} \\
            0 \ \ \ \text{otherwise.}
        \end{cases}
    \end{equation}
    Lastly, \cref{simplicialresolutions} implies 

    $$
         \beta_{i, \m}(R/I_\Delta) = \dim \tilde H_{\deg \m - i - 1}(\Delta_{\m}; \kk) = \dim \tilde H_{|V(\Delta_\m)| - i - 1}(\Delta_{\m}; \kk)
    $$
    so the result follows by \eqref{formula}.
\end{proof}

The following results are consequences of \cref{t:topbetti}, and imply, for example, that the multigraded betti numbers of ternary graphs behave nicely, that is, top betti numbers come from the top.

\begin{corollary}\label{algcombinvs}
    Let $\Delta$ be a $\kk$-spherical complex and $I_\Delta \subseteq R = \kk[x_1,\dots,x_n]$ its Stanley-Reisner ideal. Then
    \begin{itemize}
        \item $\pd(\Delta) = \pd(R/I_\Delta)$
        \item $\dep(\Delta) = \dep(R/I_\Delta)$
    \end{itemize}
\end{corollary}

\begin{proof}
    We split the proof in two cases, when $\Delta$ is acyclic, and when it is not.

    \begin{enumerate}
        \item Assume $\Delta$ has the homology of a $d_\Delta$-dimensional sphere and let $\Sigma = \Delta\lr{A}{B}$ be any induced complex of $\Delta$ such that $\Sigma$ has the homology of a $d_\Sigma$-dimensional sphere. Note that for any complex $\Sigma = \Delta\lr{A}{B}$ we have $|V(\Sigma)| + |A| + |B| \leq |V(\Delta)|$, with equality for example if $A = \emptyset$.
        
        By the left hand inequality of \cref{t:topbetti} and the remark above, we have the following inequalities.
        
        \begin{align*}
            d_\Delta - |A| - |B| & \leq d_\Sigma \\
            d_\Delta - d_\Sigma & \leq |A| + |B| \leq |V(\Delta)| - |V(\Sigma)| \\
            |V(\Sigma)| - d_\Sigma & \leq |V(\Delta)| - d_\Delta \\
            \pd(\Sigma) & \leq \pd(\Delta)
        \end{align*}
        
        By \cref{mainthm} and the last inequality above, we see that 
                $$
        \pd(R/I_\Delta) = \max (i \st  \beta_i(R/I_\Delta) \neq 0) = \max(\pd(\Sigma) \st \sign(\Sigma) = -1) = \pd(\Delta).
        $$
        From the Auslander-Buchsbaum formulas
        $$
            \pd(R/I_\Delta) + \dep(R/I_\Delta) = \dim R = |V(\Delta)| = \pd(\Delta) + \dep(\Delta)
        $$
        so the second item also holds.

        \item If $\Delta$ is acyclic, then by \cref{mainthm} and by definition we have

        $$
            \pd(\Delta) = \max(\pd(\Sigma) \st \sign(\Sigma) = -1) = \max(i \st  \beta_{i}(R/I_\Delta) \neq 0) = \pd(R/I_\Delta)
        $$
        so the first item holds. The second item then holds by the Auslander-Buchsbaum formulas.
    \end{enumerate}
\end{proof}

\begin{corollary}\label{depthreg}
    Let $\Delta$ be a non-acyclic $\kk$-spherical complex. Then 
    $$
        \reg(R/I_\Delta) = \dep(R/I_\Delta).
    $$
\end{corollary}

\begin{proof}
    By \cref{mainthm} and the Auslander-Buchsbaum formulas we have 
    \begin{align*}
        \reg(R/I_\Delta) & = \max(\deg \m - i \st  \beta_{i, \m}(R/I_\Delta) \neq 0) \\
                     & = \max(|V(\Sigma)| - \pd(\Sigma) \st \sign(\Sigma) = -1 \qand \Sigma \mbox{ induced}) \\
                     & = \max(\dep(\Sigma) \st \sign(\Sigma) = -1 \qand \Sigma \mbox{ induced})
    \end{align*}

    From the right hand inequality of \cref{t:topbetti}, and since $\Delta$ is not acyclic we conclude 
    $$
        \reg(R/I_\Delta) = \max(\dep(\Sigma) \st \sign(\Sigma) = -1) = \dep(\Delta) = \dep(R/I_\Delta).
    $$
\end{proof}

\begin{example}\label{illustrativeexample}
    Let  $I(G) = (x_1 x_2, x_2 x_3, x_3 x_4, x_4 x_1, x_3 x_5, x_1 x_6) \subseteq S = \kk[x_1,\dots, x_6]$ be the edge ideal of a square with two pendant leaves at nonadjacent vertices, and $\Delta$ its independence complex. The sequence
    $$
        \Delta \supseteq  \Delta\lr{1}{\emptyset} \supseteq  \Delta\lr{1, 5}{\emptyset} = \emptyset
    $$
    is a maximal filtration of $\Delta$, so we conclude $\reg(R/I(G)) = \dep(R/I(G)) = 2$.
\end{example}

\begin{remark}[\textbf{Leray number and regularity}]
    It is well known that the Leray number of $\Delta$ is related to the Castelnuovo-Mumford regularity of $R/I_\Delta$ (see~\cite{MR2259083}). More specifically, 
    $$
        L_\kk(\Delta) = \reg(R/I_\Delta).
    $$
    In particular, \cref{depthreg} can be seen as an algebraic analogue of~\cref{c:lerayspherical}.
\end{remark}

%%%%%%%%%%%%%%%%%%%%%%%%%%%%%%%%%%%%%%%%%%%%%%%%%%%%%%%%%%%%%%%%%
\section{Concluding remarks and questions}\label{s:questions}
%%%%%%%%%%%%%%%%%%%%%%%%%%%%%%%%%%%%%%%%%%%%%%%%%%%%%%%%%%%%%%%%%

\begin{remark}[\textbf{Filtrations, star packings, regularity and regular sequences}]\label{r:filtrations-star-packing}

    Each step of the filtration  in \cref{illustrativeexample} corresponds to removing a star from the graph. Therefore the filtration can be understood as a \emph{star packing} of $G$, as in \cite{MR3213742}, where the authors define star packings satisfying specific properties (see the introduction of \cite{MR3213742} for the exact definition) and show that the regularity of edge ideals of graphs is bounded above by an invariant based on maximizing such star packings. 

    The authors in \cite{MR4093701} also use star packings satisfying specific properties, but instead of an upper bound for regularity, their result gives a lower bound for the depth. 
\end{remark}

To showcase how filtrations could be used to simplify the search for homological invariants of graphs, we 
apply filtrations in \cref{p:Tran} to recover a result about unicyclic graphs that was recently proved in~\cite{tran}. 
In what follows, we write $G\lr{A}{B}$ instead of $\Ind(G)\lr{A}{B}$.
The statement we prove will assume the graph is ternary as it is most closely related to the theme of this paper, but the same argument can be made to prove the non-ternary case. To do so we need the following statements. 

\begin{lemma}\label{l:1}  Let $G$ be a ternary graph $G$ with two connected components $H_1$ and  $H_2$. Then  $\Ind(G)$ is not contractible if and only if  $\Ind(H_1)$ and $\Ind(H_2)$ are not contractible. In this case, $\Ind(G)\cong S^{n + m + 1}$ where $\Ind(H_1) \cong S^{n}$ and $\Ind(H_2) \cong S^m$. 
\end{lemma}
\begin{proof}
    The proof follows directly from the fact that the join of two spheres is a sphere (see, e.g.~\cite[Exercise~0.18]{MR1867354}), and \cref{ternaryintro}.
\end{proof}

\begin{lemma}[\cite{MR1713484}]\label{l:2} The independence complex of the path graph $G$ on $n$ vertices satisfies the following equalities.
$$
    \Ind(G) \cong \begin{cases}
        S^{k - 1} \quad \ \ \ \text{ if $n = 3k$} \\
        \text{a point} \quad \text{ if $n = 3k + 1$} \\
        S^{k} \quad \quad \ \ \ \text{ if $n = 3k + 2$}
    \end{cases}
$$
\end{lemma}

\begin{proposition}\label{p:Tran}
    Let $G_{n, m}$ be a unicyclic graph as in \cref{f:unicyclic} where $n\geq 3$, $m \geq 1$,   
$3\not | n$, and $n \neq m$ mod $3$. Then 
\begin{align*}
\dep(R/I(G_{n,m}) = \reg(R/I(G_{n,m})) & = 
\begin{cases}
\frac{n - 1}{3} + \frac{m}{3} & \text{$n = 1$, $m = 0$ mod $3$} \\
    \frac{n - 1}{3} + \frac{m - 2}{3} + 1 &  \text{$n = 1$, $m = 2$ mod $3$} \\
    \frac{n - 2}{3} + \frac{m}{3} + 1 &  \text{$n = 2$, $m = 0$ mod $3$} \\
    \frac{n - 2}{3} + \frac{m - 1}{3} + 1 &  \text{$n = 2$, $m = 1$ mod $3$}.
\end{cases}
\end{align*}
\begin{figure}    
\begin{center}
\tikzset{every picture/.style={line width=0.75pt}} %set default line width to 0.75pt        
\begin{tikzpicture}[x=0.75pt,y=0.75pt,yscale=-1,xscale=1]
%Straight Lines [id:da00866269225999039] 
\draw    (209.91,119.56) -- (262,86) ;
\draw [shift={(262,86)}, rotate = 327.2] [color={rgb, 255:red, 0; green, 0; blue, 0 }  ][fill={rgb, 255:red, 0; green, 0; blue, 0 }  ][line width=0.75]      (0, 0) circle [x radius= 3.35, y radius= 3.35]   ;
\draw [shift={(209.91,119.56)}, rotate = 327.2] [color={rgb, 255:red, 0; green, 0; blue, 0 }  ][fill={rgb, 255:red, 0; green, 0; blue, 0 }  ][line width=0.75]      (0, 0) circle [x radius= 3.35, y radius= 3.35]   ;
%Straight Lines [id:da27486808228249204] 
\draw    (307,85) -- (262,86) ;
\draw [shift={(262,86)}, rotate = 178.73] [color={rgb, 255:red, 0; green, 0; blue, 0 }  ][fill={rgb, 255:red, 0; green, 0; blue, 0 }  ][line width=0.75]      (0, 0) circle [x radius= 3.35, y radius= 3.35]   ;
\draw [shift={(307,85)}, rotate = 178.73] [color={rgb, 255:red, 0; green, 0; blue, 0 }  ][fill={rgb, 255:red, 0; green, 0; blue, 0 }  ][line width=0.75]      (0, 0) circle [x radius= 3.35, y radius= 3.35]   ;
%Straight Lines [id:da30071092835173885] 
\draw    (347,86) -- (307,85) ;
\draw [shift={(307,85)}, rotate = 181.43] [color={rgb, 255:red, 0; green, 0; blue, 0 }  ][fill={rgb, 255:red, 0; green, 0; blue, 0 }  ][line width=0.75]      (0, 0) circle [x radius= 3.35, y radius= 3.35]   ;
\draw [shift={(347,86)}, rotate = 181.43] [color={rgb, 255:red, 0; green, 0; blue, 0 }  ][fill={rgb, 255:red, 0; green, 0; blue, 0 }  ][line width=0.75]      (0, 0) circle [x radius= 3.35, y radius= 3.35]   ;
%Straight Lines [id:da43479562673647365] 
\draw    (395.91,87.56) -- (444,87) ;
\draw [shift={(444,87)}, rotate = 359.33] [color={rgb, 255:red, 0; green, 0; blue, 0 }  ][fill={rgb, 255:red, 0; green, 0; blue, 0 }  ][line width=0.75]      (0, 0) circle [x radius= 3.35, y radius= 3.35]   ;
\draw [shift={(395.91,87.56)}, rotate = 359.33] [color={rgb, 255:red, 0; green, 0; blue, 0 }  ][fill={rgb, 255:red, 0; green, 0; blue, 0 }  ][line width=0.75]      (0, 0) circle [x radius= 3.35, y radius= 3.35]   ;
%Straight Lines [id:da7737331895573971] 
\draw    (170.41,90.87) -- (209.91,119.56) ;
\draw [shift={(209.91,119.56)}, rotate = 36] [color={rgb, 255:red, 0; green, 0; blue, 0 }  ][fill={rgb, 255:red, 0; green, 0; blue, 0 }  ][line width=0.75]      (0, 0) circle [x radius= 3.35, y radius= 3.35]   ;
\draw [shift={(170.41,90.87)}, rotate = 36] [color={rgb, 255:red, 0; green, 0; blue, 0 }  ][fill={rgb, 255:red, 0; green, 0; blue, 0 }  ][line width=0.75]      (0, 0) circle [x radius= 3.35, y radius= 3.35]   ;
%Straight Lines [id:da6695569072282235] 
\draw    (225,166) -- (209.91,119.56) ;
\draw [shift={(209.91,119.56)}, rotate = 252] [color={rgb, 255:red, 0; green, 0; blue, 0 }  ][fill={rgb, 255:red, 0; green, 0; blue, 0 }  ][line width=0.75]      (0, 0) circle [x radius= 3.35, y radius= 3.35]   ;
\draw [shift={(225,166)}, rotate = 252] [color={rgb, 255:red, 0; green, 0; blue, 0 }  ][fill={rgb, 255:red, 0; green, 0; blue, 0 }  ][line width=0.75]      (0, 0) circle [x radius= 3.35, y radius= 3.35]   ;
%Straight Lines [id:da5958493812254566] 
\draw    (209.91,212.44) -- (170.41,241.13) ;
\draw [shift={(170.41,241.13)}, rotate = 144] [color={rgb, 255:red, 0; green, 0; blue, 0 }  ][fill={rgb, 255:red, 0; green, 0; blue, 0 }  ][line width=0.75]      (0, 0) circle [x radius= 3.35, y radius= 3.35]   ;
\draw [shift={(209.91,212.44)}, rotate = 144] [color={rgb, 255:red, 0; green, 0; blue, 0 }  ][fill={rgb, 255:red, 0; green, 0; blue, 0 }  ][line width=0.75]      (0, 0) circle [x radius= 3.35, y radius= 3.35]   ;
%Straight Lines [id:da9016835928241289] 
\draw    (121.59,90.87) -- (82.09,119.56) ;
\draw [shift={(82.09,119.56)}, rotate = 144] [color={rgb, 255:red, 0; green, 0; blue, 0 }  ][fill={rgb, 255:red, 0; green, 0; blue, 0 }  ][line width=0.75]      (0, 0) circle [x radius= 3.35, y radius= 3.35]   ;
\draw [shift={(121.59,90.87)}, rotate = 144] [color={rgb, 255:red, 0; green, 0; blue, 0 }  ][fill={rgb, 255:red, 0; green, 0; blue, 0 }  ][line width=0.75]      (0, 0) circle [x radius= 3.35, y radius= 3.35]   ;
%Straight Lines [id:da6499782956260034] 
\draw    (67,166) -- (82.09,212.44) ;
\draw [shift={(82.09,212.44)}, rotate = 72] [color={rgb, 255:red, 0; green, 0; blue, 0 }  ][fill={rgb, 255:red, 0; green, 0; blue, 0 }  ][line width=0.75]      (0, 0) circle [x radius= 3.35, y radius= 3.35]   ;
\draw [shift={(67,166)}, rotate = 72] [color={rgb, 255:red, 0; green, 0; blue, 0 }  ][fill={rgb, 255:red, 0; green, 0; blue, 0 }  ][line width=0.75]      (0, 0) circle [x radius= 3.35, y radius= 3.35]   ;
%Straight Lines [id:da15963538194594307] 
\draw    (121.59,241.13) -- (170.41,241.13) ;
\draw [shift={(170.41,241.13)}, rotate = 360] [color={rgb, 255:red, 0; green, 0; blue, 0 }  ][fill={rgb, 255:red, 0; green, 0; blue, 0 }  ][line width=0.75]      (0, 0) circle [x radius= 3.35, y radius= 3.35]   ;
\draw [shift={(121.59,241.13)}, rotate = 360] [color={rgb, 255:red, 0; green, 0; blue, 0 }  ][fill={rgb, 255:red, 0; green, 0; blue, 0 }  ][line width=0.75]      (0, 0) circle [x radius= 3.35, y radius= 3.35]   ;
%Straight Lines [id:da03798613132115025] 
\draw    (209.91,212.44) -- (225,166) ;
\draw [shift={(225,166)}, rotate = 288] [color={rgb, 255:red, 0; green, 0; blue, 0 }  ][fill={rgb, 255:red, 0; green, 0; blue, 0 }  ][line width=0.75]      (0, 0) circle [x radius= 3.35, y radius= 3.35]   ;
\draw [shift={(209.91,212.44)}, rotate = 288] [color={rgb, 255:red, 0; green, 0; blue, 0 }  ][fill={rgb, 255:red, 0; green, 0; blue, 0 }  ][line width=0.75]      (0, 0) circle [x radius= 3.35, y radius= 3.35]   ;
%Straight Lines [id:da30375901969093966] 
\draw    (121.59,90.87) -- (170.41,90.87) ;
\draw [shift={(170.41,90.87)}, rotate = 360] [color={rgb, 255:red, 0; green, 0; blue, 0 }  ][fill={rgb, 255:red, 0; green, 0; blue, 0 }  ][line width=0.75]      (0, 0) circle [x radius= 3.35, y radius= 3.35]   ;
\draw [shift={(121.59,90.87)}, rotate = 360] [color={rgb, 255:red, 0; green, 0; blue, 0 }  ][fill={rgb, 255:red, 0; green, 0; blue, 0 }  ][line width=0.75]      (0, 0) circle [x radius= 3.35, y radius= 3.35]   ;
%Straight Lines [id:da46859923904677636] 
\draw    (67,166) -- (82.09,119.56) ;
\draw [shift={(82.09,119.56)}, rotate = 288] [color={rgb, 255:red, 0; green, 0; blue, 0 }  ][fill={rgb, 255:red, 0; green, 0; blue, 0 }  ][line width=0.75]      (0, 0) circle [x radius= 3.35, y radius= 3.35]   ;
\draw [shift={(67,166)}, rotate = 288] [color={rgb, 255:red, 0; green, 0; blue, 0 }  ][fill={rgb, 255:red, 0; green, 0; blue, 0 }  ][line width=0.75]      (0, 0) circle [x radius= 3.35, y radius= 3.35]   ;
%Straight Lines [id:da19577268135657233] 
\draw    (82.09,212.44) -- (121.59,241.13) ;
\draw [shift={(121.59,241.13)}, rotate = 36] [color={rgb, 255:red, 0; green, 0; blue, 0 }  ][fill={rgb, 255:red, 0; green, 0; blue, 0 }  ][line width=0.75]      (0, 0) circle [x radius= 3.35, y radius= 3.35]   ;
\draw [shift={(82.09,212.44)}, rotate = 36] [color={rgb, 255:red, 0; green, 0; blue, 0 }  ][fill={rgb, 255:red, 0; green, 0; blue, 0 }  ][line width=0.75]      (0, 0) circle [x radius= 3.35, y radius= 3.35]   ;
% Text Node
\draw (361,84) node [anchor=north west][inner sep=0.75pt]   [align=left] {$\displaystyle \dotsc $};
% Text Node
\draw (188,120) node [anchor=north west][inner sep=0.75pt]   [align=left] {$\displaystyle x_{1}$};
% Text Node
\draw (227,169) node [anchor=north west][inner sep=0.75pt]   [align=left] {$\displaystyle x_{2}$};
% Text Node
\draw (211.91,215.44) node [anchor=north west][inner sep=0.75pt]   [align=left] {$\displaystyle x_{3}$};
% Text Node
\draw (45,106) node [anchor=north west][inner sep=0.75pt]   [align=left] {$\displaystyle x_{n-2}$};
% Text Node
\draw (97,66) node [anchor=north west][inner sep=0.75pt]   [align=left] {$\displaystyle x_{n-1}$};
% Text Node
\draw (170,69) node [anchor=north west][inner sep=0.75pt]   [align=left] {$\displaystyle x_{n}$};
% Text Node
\draw (172.41,244.13) node [anchor=north west][inner sep=0.75pt]   [align=left] {$\displaystyle x_{4}$};
% Text Node
\draw (111,245) node [anchor=north west][inner sep=0.75pt]   [align=left] {$\displaystyle x_{5}$};
% Text Node
\draw (244,62) node [anchor=north west][inner sep=0.75pt]   [align=left] {$\displaystyle y_{1}$};
% Text Node
\draw (299,58) node [anchor=north west][inner sep=0.75pt]   [align=left] {$\displaystyle y_{2}$};
% Text Node
\draw (337,60) node [anchor=north west][inner sep=0.75pt]   [align=left] {$\displaystyle y_{3}$};
% Text Node
\draw (382,61) node [anchor=north west][inner sep=0.75pt]   [align=left] {$\displaystyle y_{m-1}$};
% Text Node
\draw (436,61) node [anchor=north west][inner sep=0.75pt]   [align=left] {$\displaystyle y_{m}$};
\end{tikzpicture}
\end{center}
\caption{The unicyclic graph in~\cref{p:Tran}}\label{f:unicyclic}
\end{figure}
\end{proposition}

\begin{proof}
Note that $G_{n,m}\lr{x_1}{\emptyset}$ and $G_{n,m}\lr{\emptyset}{x_1}$ are both disconnected graphs where every connected component is a path graph. By \cref{l:2}, if a connected component of a ternary graph is a path graph on $3k + 1$ vertices, its independence complex is contractible.
\begin{enumerate}
    \item If $n = 3k + 1$ and $m = 3l$, then the two connected components of $G_{n,m}\lr{x_1}{\emptyset}$ have $3(k - 1) + 1$ and $m - 1$ vertices respectively, so $d\lr{x_1}{\emptyset} = \ast$ by \cref{l:2,l:1}. On the other hand, $G_{n,m}\lr{\emptyset}{x_1}$ has two connected components that are paths with $3k$ and $3l$ vertices respectively. By \cref{l:2}, we conclude $\Ind(G_{n,m}\lr{\emptyset}{x_1}) \cong S^{k + l - 1}$ and by \cref{triangles} we conclude $\Ind(G_{n,m}) \cong S^{k + l - 1}$.
    \item If $n = 3k + 1$ and $m = 3l + 2$, then just as in the previous case $d\lr{x_1}{\emptyset} = \ast$ and $d\lr{\emptyset}{x_1} \neq \ast$, so $d\lr{\emptyset}{\emptyset} \neq \ast$. Since $m = 3l + 2$ and $G_{n,m}\lr{\emptyset}{x_1}$ has one connected component with $3k$, and one with $3l + 2$ vertices, by \cref{l:1,l:2} we conclude $d\lr{\emptyset}{x_1} = k + l$ and by \cref{triangles}, we conclude $\Ind(G_{n,m}) \cong S^{k + l}$.
\end{enumerate}
By similar arguments, if $n = 3k + 2$ and $m = 3k$, we conclude $\Ind(G_{n,m}) \cong S^{k + l}$ and if $n = 3k + 2$ and $m = 3k + 1$, $\Ind(G_{n,m}) \cong S^{k + l}$.

In all the cases above, \cref{ternarytype}, \cref{algcombinvs,depthreg} imply $\dep(R/I(G)) = \reg(R/I(G))$. The computations above together with \cref{ternarytype,algcombinvs} then imply the equalities as desired.
\end{proof} 

We end the paper with some natural questions that follow from our work.

It is known that for a simplicial complex $\Delta$ that is not a cone over one of its vertices, $R/I_\Delta$ being Gorenstein implies $\Delta$ is not acyclic. \cref{t:topbetti} restricts the possible nonzero homology groups of $\Sigma$, where $\Sigma$ are induced subgraphs of $\Delta$. One question to consider is the following.

\begin{question}
    Let $\Delta$ be a spherical complex that is not a cone over any of its vertices. Is it true that
    $$
        \text{$R/I_\Delta$ is Gorenstein} \iff \text{ $\Delta$ is not acyclic and $R/I_\Delta$ is Cohen-Macaulay?}
    $$

\end{question}

Note that $\implies$ is clear by \cite[Theorem 5.1, Chapter 2]{stanleybook}.

\begin{question}\label{q:star-packing}
    To which extent can previous bounds on algebraic invariants such as projective dimension, depth and regularity be extended to equalities for ternary graphs? In particular, if a ternary graph $G$ has a noncontractible independence complex, let $\zeta(G)$ be the star packing invariant defined in \cite{MR3213742}, is it true that
    $$
        \dep(R/I(G)) = \dep(\Ind(G)) = \zeta(G)?
    $$
\end{question}

{\bf Acknowledgments.}
We would like to thank Susan Morey for pointing us to the results in~\cite{MR4093701}, and Seyed Amin Seyed Fakhari for helpful comments. Faridi's research is supported by
NSERC Discovery Grant 2023-05929.

%%%%%%%%%%%%%%%%%%%%%%%%%%%%%%%%%%%%%%%%%%%%%%%%%%%%%%%%
\bibliography{bibliography.bib}
\bibliographystyle{plain}
%%%%%%%%%%%%%%%%%%%%%%%%%%%%%%%%%%%%%%%%%%%%%%%%%%%%%%%%

%%%%%%%%%%% END OF SPHERICAL PAPER%%%%%%%%%%%%%%%%%%%%%%

\end{document}